\documentclass[a4,11pt]{article}

\usepackage{fullpage,amsfonts,amsmath,amsthm,amssymb,mathtools}
\usepackage{tikz,color,algorithmic,algorithm,multirow,graphicx,caption,subcaption}
\usepackage{pdflscape,multirow}
\theoremstyle{plain}
\newtheorem{theorem}{Theorem}
\newtheorem{corollary}[theorem]{Corollary}
\newtheorem{lemma}[theorem]{Lemma}
\newtheorem{remark}[theorem]{Remark}
\theoremstyle{definition}

\newtheorem{definition}[theorem]{Definition}

\newcommand{\R}{\mathbf{R}}
\newcommand{\oo}{\mathcal{O}}

\newcommand{\V}{V_+}
\newcommand{\W}{W_+}
\newcommand{\U}{U_+}

\newcommand{\ff}{\mathbf{f}}

\newcommand{\ffp}{\mathbf{f}_+}

\newcommand{\df}{\delta\mathbf{f}}
\newcommand{\dfp}{\delta\mathbf{f}_+}
\newcommand{\dfpp}{\delta\mathbf{f}_+'}

\newcommand{\vzp}{z_+}

\newcommand{\vz}{z}

\newcommand{\eqdef}{:=}

\newcommand{\hadprod}{\odot}

\newcommand{\mat}[1]{\begin{bmatrix}#1\end{bmatrix}}
\usetikzlibrary{arrows}

\title{Gradient and Hessian approximations in Derivative Free Optimization}
\author{Ian D. Coope and Rachael Tappenden}
\begin{document}
  \maketitle

\begin{abstract}
This work investigates finite differences and the use of interpolation models to obtain approximations to the first and second derivatives of a function. Here, it is shown that if a particular set of points is used in the interpolation  model, then the solution to the associated linear system (i.e., approximations to the gradient and diagonal of the Hessian) can be obtained in $\oo(n)$ computations, which is the same cost as finite differences, and is a saving over the $\oo(n^3)$ cost when solving a general unstructured linear system. Moreover, if the interpolation points are formed using a `regular minimal positive basis', then the error bound for the gradient approximation is the same as for a finite differences approximation. Numerical experiments are presented that show how the derivative estimates can be employed within an existing derivative free optimization algorithm, thus demonstrating one of the potential practical uses of these derivative approximations.
\end{abstract}


\section{Introduction}
This work considers unconstrained optimization problems of the form
\begin{eqnarray}\label{eq:prob}
  \min_{x\in\R^n} f(x),
\end{eqnarray}
where $f:\R^n \to \R$ is smooth, but the derivatives of $f$ are unavailable. This is a typical derivative free optimization setting. Problems of this form arise, for example when the function is smooth but it is computationally impractical to obtain derivative information, or when the function itself is not known explicitly although function values are available from a black-box or oracle, (for example via physical measurements). 

A key ingredient of many algorithms for solving \eqref{eq:prob} in a derivative free setting is an estimate of the gradient of $f$. Sometimes gradients are used to determine search directions (for example, in implicit filtering algorithms \cite{Cocchi2018,Gilmore1995b,Gilmore1995a}), and sometimes they are used in the definition of an algorithm's stopping criteria. Several techniques exist for determining estimates to derivatives, including automatic differentiation \cite{Hoffmann2016,Baydin2018,Margossian2019}, finite difference approximations, interpolation models \cite{Conn1996,Conn1997}, Gaussian smoothing \cite{Maggiar2018,Nesterov2017}, and smoothing on the unit sphere \cite{Fazel2018}. This work investigates finite differences and interpolation models, so henceforth, we focus on these approaches.

A well known, simple, and widely applicable approach to estimating derivatives is finite difference approximations (or finite differences). Finite differences are so called because they provide estimates to derivatives by taking the difference between function values at certain known points. For example, for a function $f:\R^n\to \R$, the forward differences approximation $g(x)$ to the (first) derivative is
\begin{equation}
  [g(x)]_i = \frac{f(x+h e_i) - f(x)}{h}, \qquad \text{for} \; i = 1,\dots,n,
\end{equation}
which is an $\oo(h)$ accurate approximation. Many other (well known) formulations exist, including approximations that provide a higher accuracy estimates of the gradient, as well as formulations for higher order derivatives. The standard finite difference approximation formulations can be used whenever function values evaluated along coordinate directions are available. The formulations involve minimal computations ($\oo(n)$ operations), and they have low memory requirements because one does not need to store the coordinate directions, but finite differences can suffer from cancellation errors.

Another approach to approximating derivatives is to use interpolation models \cite{Conn2009,Wild2013,Hare2015}. Interpolation models can be viewed as a generalization of finite differences, where, instead of being restricted to the coordinate directions, function evaluations along a general set of directions are utilized, and derivative approximations are found by solving an associated system of equations (i.e., minimizing the interpolation model). Linear models lead to approximations to the gradient \cite{Berahas2019}, while quadratic models provide approximations to the gradient and Hessian \cite{Conn1996}. Much research has focused on the `quality' of the set of directions, because this impacts the accuracy of the derivative estimates \cite{Bandeira2012,Conn2008,Conn2008b,Fasano2009}.

The purpose of this work is to show that, if one uses a (diagonal) quadratic interpolation model with carefully chosen interpolation directions and $2n+2$ function values, then approximations to the gradient and diagonal of the Hessian can be obtained in $\oo(n)$ computations and with only $\oo(1)$ vectors of storage. The solution to this specific instance of an interpolation model can be thought of, in some sense, as a general finite difference type approximation to the gradient and diagonal of the Hessian, because the solution to the interpolation model is a formula that involves sums/differences of function values at the interpolation points. Although the focus of this work is on derivative estimates and not algorithmic development, the estimates here are widely applicable, and could be employed by many algorithms that require gradient approximations. Moreover, the diagonal Hessian approximation is readily available, and could be used as preconditioner within an algorithm. (See the numerical experiments in Section~\ref{S:numerics}, which include a derivative free preconditioned conjugate gradients algorithm.)

It remains to note that the approaches mentioned above all require access to function values, and the cost of evaluating a function in a derivative free optimization setting can vary greatly depending upon the particular application being considered. It is up to the user to determine their function evaluation budget, and the approaches in this work require between $n+1$ and $2n+2$ function evaluations to generate derivative approximations.

\subsection{Summary of Contributions}

The contributions of this work are listed now.
\begin{enumerate}
  \item \emph{Derivative estimates.} This work builds a quadratic interpolation model (with a diagonal approximation to the Hessian), and shows that if a certain set of interpolation directions are used in the quadratic model, with $2n+2$ associated function evaluations, then approximations to the gradient $g(x)$ and diagonal of the Hessian $D(x)$ can be found in $\oo(n)$ operations. The set of interpolation directions is referred to as a regular minimal positive basis, and these directions need not be stored in full, but can be computed on-the-fly from $\oo(1)$ stored vectors.
  \item \emph{An error bound.} We confirm that the gradient approximations developed in this work are $\oo(h^2)$ accurate, where $h$ is the `sampling radius'. (See Section~\ref{S:errorbounds}.)
  \item \emph{Application.} We provide an example of how the gradient estimates developed in this work could be used in practice. In particular, we take the frame based, preconditioned conjugate gradients algorithm developed in \cite{Coope2002}, which requires estimates of the gradient and a diagonal preconditioner, and employ the estimates $g(x)$ and $D(x)$ developed in this work. (The original paper \cite{Coope2002} employed the standard central difference approximations to the gradient and diagonal of the Hessian.)
\end{enumerate}

\subsection{Notation}

Let $e\in \R^N$ denote the vector of all ones and let $e_j\in \R^N$ for $j=1,\dots,N$ denote the standard unit basis vectors in $\R^N$, i.e., $e_j$ is the $j$th column of the $N\times N$ identity matrix $I$. (Usually we will take $N = n$ or $N = n+1$ in this work.)

Consider the points $x_1,\dots,x_{n+1} \in \R^n$. Define $f_j \eqdef f(x_j)$ for $j = 1,\dots,n+1$, i.e., $f_j$ denotes the function value at the point $x_j$. It is convenient to define the following vectors, which contain the function values:
  \begin{equation}\label{fvec}
    \ffp = \mat{\ff\\f_{n+1}}\in \R^{n+1} \quad \text{where}\quad \ff \eqdef \mat{f_1\\\vdots\\f_n} \in \R^n.
  \end{equation}

Similarly, for the points $x_1',\dots,x_{n+1}' \in \R^n$, define $f_j' \eqdef f(x_j')$ for all $j$, (i.e., $f_j'$ denotes the function value at the point $x_j'$), and
  \begin{equation}\label{fpvec}
    \ffp' = \mat{\ff'\\f_{n+1}'}\in \R^{n+1} \quad \text{where}\quad \ff' \eqdef \mat{f_1'\\\vdots\\f_n'} \in \R^n.
  \end{equation}
It will also be convenient to define the vectors
\begin{equation}\label{eq:df}
 \df \eqdef \ff - f(x) e \in\R^n, \qquad \df' \eqdef \ff' - f(x) e \in\R^n
\end{equation}
and
\begin{equation}\label{eq:dfp}
 \dfp \eqdef \ffp - f(x) e \in\R^{n+1}, \qquad \dfp' \eqdef \ffp' - f(x) e \in\R^{n+1}.
\end{equation}
Usually, vectors are assumed to lie in $\R^n$, while the subscript `$+$' denotes that the vector has an additional entry, so is an element of $\R^{n+1}$ (recall the notation above). Matrices are assumed to be elements of $\R^{n\times n}$, while the subscript `$+$' denotes an additional column, so that the matrix is an element of $\R^{n \times (n+1)}$.

The symbol `$\hadprod$' denotes the Hadamard product, i.e., given $x,y\in\R^n$, $x \hadprod y \in\R^n$ where $(x \hadprod y)_i \eqdef x_iy_i$.

\subsection{Paper Outline}

This paper is organised as follows. Section~\ref{S:finitedifferences} introduces the (diagonal) quadratic interpolation model employed in this work, and shows how derivative approximations can be found using the model. Section~\ref{S:computationviainterpolation}, describes how approximations to the gradient $g(x)$ and the diagonal of the Hessian $D(x)$ can be constructed from the model in $\oo(n)$ computations when the interpolation directions are chosen to be the coordinate basis, and also when they are chosen to be a `regular basis' (see Section~\ref{S:regularbasis}). Many algorithms in derivative free optimization are based upon simplices, which can be generated from a minimal positive basis. Thus, in Section~\ref{S:interp_minbasis}, approximations to the gradient $g(x)$ and the diagonal of the Hessian $D(x)$ that can be computed in $\oo(n)$ are presented when the interpolation points are formed from either a coordinate minimal positive basis, or a regular minimal positive basis (see Section~\ref{S:interp_regminposbasis}). It will also be shown that the regular basis and the regular minimal positive basis are closely related to a regular simplex, hence the terminology `regular' (see Section~\ref{S:simplex}). In Section~\ref{S:errorbounds} an error bound is presented to confirm that the gradient approximations developed in this work are $\oo(h^2)$ accurate, where $|h|$ is the sampling radius. Section~\ref{S:linearmodel} summarizes existing work on linear models, while numerical experiments are presented in Section~\ref{S:numerics} to demonstrate the accuracy and applicability of these results.

\section{Derivative Estimation via Interpolation Models} \label{S:finitedifferences}

This section describes how to approximate derivatives of a function via interpolation models. The Taylor series for a function $f$ about a point $x\in\R^n$ is
\begin{equation}\label{eq:Taylorserieswithx}
  f(y) = f(x) + (y-x)^T\nabla f(x) + \tfrac12(y-x)^T\nabla^2 f(x) (y-x) + \oo(\|y-x\|_2^3).
\end{equation}
Let $g(x)$ denote an approximation to the gradient $\nabla f(x)$ (i.e., $g(x) \approx \nabla f(x)$), and let $H(x)$ denote an approximation to the Hessian $\nabla^2 f(x)$ (i.e., $H(x) \approx \nabla^2 f(x)$).
Then \eqref{eq:Taylorserieswithx} becomes
\begin{equation}\label{eq:Taylorf}
  f(y) \approx f(x) + (y-x)^Tg(x) + \tfrac12(y-x)^T H(x) (y-x).
\end{equation}
In this work, only a diagonal approximation to the Hessian is considered, so define
\begin{equation}\label{eq:D}
  D(x) \eqdef \mat{d_1 & & \\ &\ddots &\\ & & d_n}\in \R^{n\times n},\qquad d(x) \eqdef \mat{d_1\\ \vdots\\ d_n}\in \R^{n},
\end{equation}
where $D(x) \approx {\rm diag}(\nabla^2 f(x))$. Combining \eqref{eq:Taylorf} and \eqref{eq:D} leads to the (diagonal) quadratic model
\begin{equation}\label{eq:quadraticmodel}
  m(y) = f(x) + (y-x)^Tg(x) + \tfrac12(y-x)^T D(x) (y-x).
\end{equation}
The motivation for this diagonal quadratic model is as follows. It is well known that, if one uses a linear model with $n$ (or $n+1$) interpolation points (i.e., $n$ or $n+1$ function evaluations), then an estimate of the gradient is available that is $\oo(h)$ accurate (see also Section~\ref{S:linearmodel} which summarises results for linear models). For some situations, an $\oo(h)$ accurate gradient is not of sufficient accuracy to make algorithmic progress, so an $\oo(h^2)$ accurate gradient is used instead. In this case, one can still use a linear model, but $2n$ (or $2n+2$) interpolation points (i.e., $2n$ (or $2n+2$) function evaluations) are needed. In derivative free optimization, function evaluations are often expensive to obtain, so it is prudent to `squeeze out' as much information from them as possible. Hence, suppose one had available $2n$ (or $2n+2$) function evaluations because an accurate gradient was desired. The model \eqref{eq:quadraticmodel} contains $2n$ unknowns, so $2n$ interpolation points and function values is enough to uniquely determine an $\oo(h^2)$ accurate gradient $g(x)$, \emph{as well as} an approximation to the pure second derivatives on the diagonal of $D(x)$, i.e., no additional function evaluations are required to compute the pure second derivatives. Although $D(x)$ can be determined  `for free' in terms of function evaluations (assuming $2n$ are available for the gradient estimate), we will also show that the diagonal matrix $D(x)$ can be computed in $\oo(n)$ flops, so the linear algebra costs are low. This is similar to the case for finite differences (i.e., one requires $2n(+1)$ function evaluations to compute $g(x)$ via the central differences formula, and those same function values can also be used to compute the pure second derivatives), and so this diagonal quadratic model \eqref{eq:quadraticmodel} is useful because one can obtain analogous results for a certain set of interpolation points.

The goal is to determine $g(x)$ and $D(x)$ from the interpolation model \eqref{eq:quadraticmodel} using a set of known points and the corresponding function values. Hence, one must build a sample set (a set of interpolation points). In general, one can use any number of interpolation points, although this affects the model. For a linear model (omitting the third term in \eqref{eq:quadraticmodel}), one can determine a unique solution $g(x)$ if there are $n$ interpolation/sample points with known function values at those points (see Section~\ref{S:linearmodel}). For a general quadratic model with Hessian approximation $H(x)$, one requires $(n+1)(n+2)/2$ function values to determine unique solutions $g(x)$ and $H(x)$. If one has access to more or fewer sample points/function values, then one can use, for example, a least squares approach to find approximations to $g(x)$ and $H(x)$. In the rest of this section, $2n$ interpolation points will be used, and it will become clear that this allows unique solutions $g(x)$ and $D(x)$ to the model \eqref{eq:quadraticmodel}.

Given $x$, a set of directions $\{u_1,\dots,u_n\}\subset \R^n$, and a scalar $h$ (sometimes referred to as the \emph{sampling radius}), a set of points $\{x_1,\dots,x_n\}$ is defined via the equations
\begin{eqnarray}\label{eq:definexj}
  x_j = x +  h u_j, \quad \|u_j\|_2 = \tfrac1{|h|}\|x_j - x\|_2, \qquad
  j = 1,\dots, n.
\end{eqnarray}
The model is constructed to satisfy the interpolation conditions
\begin{equation}\label{eq:interpolationconditions}
  f(x+ h u_j) = m(x+h u_j),\qquad j = 1,\dots,n.
\end{equation}
Substituting $y = x +  h u_j = x_j$ $\forall j$ into the model \eqref{eq:quadraticmodel} gives the equations
\begin{equation}\label{eq:modelatxj}
  m(x_j) = f(x) + h u_j^Tg(x) + \tfrac12 h^2 u_j^T D(x) u_j.
\end{equation}
Because $D(x)$ is diagonal, $u_j^T D(x) u_j = \sum_{i=1}^n d_i (u_j)_i^2 = (u_j\hadprod u_j)^Td(x),$ $\forall j$. Defining
\begin{eqnarray}
  U &\eqdef& \mat{u_1 & \dots &u_n}\label{eq:U}\\
  W &\eqdef& \mat{u_1 \hadprod u_1 & u_2 \hadprod u_2 & \dots &u_n \hadprod u_n},\label{eq:W}
\end{eqnarray}
where $U,W \in \R^{n\times n}$, allows \eqref{eq:modelatxj} to be written in matrix form as
\begin{eqnarray}\label{eq:syst1}
  \df \overset{\eqref{eq:df}}{=}  h U ^T g(x) + \tfrac12 h^2 W^T d(x).
\end{eqnarray}
The system \eqref{eq:syst1} is underdetermined, (there are only $n$ equations in $2n$ unknowns), and as the goal is to find unique solutions $g(x)$ and $D(x) = {\rm diag}(d(x))$, an additional $n$ equations are required. Define
\begin{eqnarray}\label{ass:sjnotsjp}
  h' \eqdef \eta h,\qquad \text{where } \eta \neq 1.
\end{eqnarray}
Now, keep the point $x$, and the set of directions $\{u_1,\dots,u_n\}\subset \R^n$ fixed and choose $h'$ satisfying \eqref{ass:sjnotsjp}. A new set of points $\{x_1',\dots,x_n'\}$ is defined via the equations
\begin{eqnarray}\label{eq:definexjprime}
  x_j' = x +  h' u_j, \quad \|u_j\|_2 = \tfrac{1}{|h'|}\|x_j' - x\|_2, \qquad
  j = 1,\dots, n.
\end{eqnarray}
The definition of $h'$ in \eqref{ass:sjnotsjp} ensures that $x_j' \neq x_j$ $\forall j$. The model is constructed to satisfy the interpolation conditions
\begin{equation}\label{eq:interpolationconditions2}
  f(x + \eta h u_j) = m(x + \eta h u_j) \qquad j = 1,\dots,n.
\end{equation}
(That is, $f(x_j) = m(x_j)$ and $f(x_j') = m(x_j')$ for $j = 1,\dots,n$.) Following similar arguments to those previously established, one arrives at the system
\begin{eqnarray}\label{eq:syst2}
  \df' &\overset{\eqref{eq:df}}{=}&  h' U^T g(x) + \tfrac1{2}(h')^2 W^T d(x)\notag\\
  &=&  \eta h U^T g(x) + \tfrac1{2}\eta^2 h^2 W^T d(x).
\end{eqnarray}

The systems \eqref{eq:syst1} and \eqref{eq:syst2} can be combined into block matrix form as follows:
\begin{eqnarray}\label{eq:systemblockmatrix}
  \mat{\df\\\df'} =
  \mat{hU^T & \tfrac12h^2 W^T\\\eta hU^T & \tfrac12\eta^2h^2 W^T}
  \mat{g(x)\\d(x)}.
\end{eqnarray}
It is clear that \eqref{eq:systemblockmatrix} represents a system of $2n$ equations in 2n unknowns. Performing block row operations on the system \eqref{eq:systemblockmatrix} (i.e., multiplying the first row by $\eta$ and subtracting the second row, and also multiplying the first row by $\eta^2$ and subtracting the second row) gives
\begin{eqnarray}\label{eq:systemblockmatrix}
  \mat{\eta\df - \df' \\\eta^2\df - \df'} =
  \mat{0 & \tfrac12h^2 (\eta-\eta^2) W^T\\ (\eta^2-\eta) hU^T & 0}
  \mat{g(x)\\d(x)}.
\end{eqnarray}
Hence, estimates of the gradient and diagonal of the Hessian can be found by solving
\begin{eqnarray}\label{eq:Findgd}
  y = h U^T g(x), \qquad \text{and}\qquad  z = \tfrac12 h^2 W^T d(x),
\end{eqnarray}
where
\begin{eqnarray}
   y &\eqdef& \tfrac1{\eta(\eta-1)}(\eta^2\df - \df')\label{eq:yS}\\
   z &\eqdef& \tfrac1{\eta(1-\eta)}(\eta\df - \df')\label{eq:zS}.
\end{eqnarray}
Questions naturally arise, such as `when do unique solutions to \eqref{eq:Findgd} exist?' and `what is the cost of computing the solutions?'. Clearly, if the matrices $U$ and $W$ have full rank, then unique solutions to \eqref{eq:Findgd} exist. The rank of $U$ and $W$, of course, depends upon how the interpolation points are chosen. If the interpolation directions $\{u_1,\dots,u_n\}$ form a basis for $\R^n$, then a unique solution $g(x)$ exists. Similarly, if the vectors $w_j = u_j\hadprod u_j$, for $j=1,\dots,n$ form a basis for $\R^n$, then a unique solution $d(x)$ exists.

Computing the solution to a general unstructured system has a cost of $\oo(n^3)$ flops. The focus of the remainder of this paper is to show that, if one selects the interpolation points/interpolation directions in a specific way, then unique solutions $g(x)$ and $d(x)$ to \eqref{eq:Findgd} exist, and moreover, the computational cost of finding the solutions remains low, at $\oo(n)$. This is the \emph{same cost as finite difference approximations} to the gradient and pure second order derivatives. Importantly, it will also be shown that this specific set of directions does not increase the memory footprint either, requiring only $\oo(1)$ vectors of storage.

In what follows, it will be useful to notice that, in the special case $\eta = -1$, \eqref{eq:yS} and \eqref{eq:zS} become
\begin{eqnarray}
   y &=& \tfrac1{2}(\df - \df'),\qquad [y]_i = \tfrac1{2}(f(x+hu_j) - f(x-hu_j))\label{eq:yh}\\
   z &=& \tfrac1{2}(\df + \df'),\qquad [z]_i = \tfrac1{2}(f(x+hu_j) + f(x-hu_j) - 2f(x))\label{eq:zh}.
\end{eqnarray}

The Sherman-Morrison formula for a nonsingular matrix $A\in \R^{n\times n}$, and vectors $u,v\in \R^n$ is used many times in this work, and is stated now for convenience:
\begin{equation}\label{eq:SMW}
  (A + uv^T)^{-1} = A^{-1} - \frac{A^{-1}uv^TA^{-1}}{1+v^TA^{-1}u}.
\end{equation}

\section{Computing derivative estimates via interpolation}
\label{S:computationviainterpolation}

This section describes how to compute solutions to the diagonal quadratic interpolation model when the coordinate basis is used to construct the interpolation points, and also when the regular basis is employed. In both cases the computational cost of determining the solution is $\oo(n)$, and when the coordinate basis is used with $\eta = -1$ in \eqref{ass:sjnotsjp}, the usual finite differences approximations to the first and second (pure) derivatives are recovered.

\subsection{Interpolation using the coordinate basis}

Here we present the solutions $g(x)$ and $d(x)$ to \eqref{eq:Findgd} when the coordinate basis $\{e_1,\dots,e_n\}\subset\R^n$ is used to generate the interpolation points. Given a point $x\in \R^n$, take $u_j = e_j$  $\forall j$, so that $U \equiv I$. Because $e_j \hadprod e_j = e_j, \;\; \forall j$, $W \equiv I$. Subsequently, the points $x_1,\dots,x_n,x_1',\dots,x_n'$ satisfy $x_j = x + h e_j$, $x_j' = x + \eta h e_j$, $j=1,\dots n$. Substituting this into \eqref{eq:Findgd}, shows that $g(x) = \tfrac1h y$ and $d(x) = \tfrac2{h^2} z$. Now, set $\eta = -1$. Using \eqref{eq:yh} in the expression for $g(x)$ gives
\begin{eqnarray*}
  g(x) = \tfrac1{2h} (\df  - \df') = \tfrac1{2h}(\ff - \ff'),
\end{eqnarray*}
or equivalently, the $i$th element of $g(x)$ is
\begin{eqnarray}\label{eq:gCFD}
  [g(x)]_i = \tfrac1{2h} (f(x+h e_i) - f(x - he_i)).
\end{eqnarray}
Notice that \eqref{eq:gCFD} is the standard central differences approximation to the gradient, which is known to be $\oo(h^2)$ accurate (see also Section~\ref{S:errorbounds}).

Similarly, substituting \eqref{eq:zh} in the expression for $d(x)$ gives
\begin{eqnarray*}
  d(x) = \tfrac{1}{h^2}\left(\df + \df'\right) = \tfrac{1}{h^2}\left(\ff + \ff' - 2 f(x) e\right),
\end{eqnarray*}
or equivalently, the $i$th element of $d(x)$ is
\begin{eqnarray}\label{eq:dCFD}
  [d(x)]_i = \tfrac1{h^2} \left(f(x + h e_i) + f(x - he_i) - 2f(x)\right),
\end{eqnarray}
which is the standard, central differences approximation to the diagonal entries of the Hessian.

This confirms that when the coordinate basis is used in the diagonal quadratic interpolation model, with $h' = -h$, then \eqref{eq:Findgd} recovers the central finite differences formulas for the gradient and pure second order derivatives.

\subsection{Interpolation using a regular basis}
\label{S:regularbasis}

This section shows how gradient and diagonal Hessian estimates can be computed in $\oo(n)$ flops if a specific basis is used to generate the interpolation points. For ease of reference, we call this a \emph{regular basis}. (It will be shown that the directions forming the regular basis are $n$ of the internal arms of a regular simplex, see Section~\ref{S:interp_minbasis}).

Define the scalars
\begin{eqnarray}\label{eq:alphagamma}
  \alpha \eqdef \sqrt{\frac{n+1}{n}},\qquad \gamma \eqdef \frac1n\left(1 - \sqrt{\frac1{n+1}}\right).
\end{eqnarray}
Let $\{v_1,\dots,v_n\}$ be the regular basis for $\R^n$, whose elements are defined by
\begin{equation}\label{eq:vj}
  v_j \eqdef \alpha(e_j - \gamma e),\qquad j = 1,\dots, n,
\end{equation}
where the basis vectors can be collected as the columns of the matrix
\begin{eqnarray}\label{eq:V}
  V \eqdef \mat{v_1&\dots&v_n} = \alpha(I - \gamma ee^T) \in \R^{n\times n}.
\end{eqnarray}
Equation \eqref{eq:vj} shows that the vectors $v_j$ are simply a multiple of $e$ (the vector of all ones) with an adjustment to the $j$th component. Thus, these vectors need not be stored explicitly; they can simply be generated on-the-fly and then discarded. Therefore, the regular basis has a low storage footprint of $\oo(1)$ vectors.

In \cite[p.569]{Coope19}, (see also page 20 and Corollary~2.6 in \cite{Conn09}) it is shown that for $V$ defined in \eqref{eq:V},
\begin{eqnarray*}
  V^2 = V^TV = \mat{
  1    & -\frac1n  & \dots &-\frac1n\\
      -\frac1n & 1 &  &\vdots\\
      \vdots &  & \ddots & -\frac1n\\
      -\frac1n & \dots   & -\frac1n & 1
  },
\end{eqnarray*}
which shows that the elements of the regular basis have unit length:
\begin{equation}\label{eq:normvj}
  \|v_j\|_2 = 1, \qquad j = 1,\dots, n.
\end{equation}
The following Lemma confirms that $\{v_1,\dots,v_n\}$ is a basis for $\R^n$, and also that $V$ is positive definite.
\begin{lemma}[Lemma~2 in \cite{Coope19}]\label{lem:Vevals}
Let $\alpha$ and $\gamma$ be defined in \eqref{eq:alphagamma}. Then $V$ in \eqref{eq:V} is nonsingular. Moreover, $V$ has $n-1$ eigenvalues equal to $\alpha$, and one eigenvalue equal to $1/\sqrt{n}$.
\end{lemma}
\begin{corollary}\label{cor:Vevals}
  Let the conditions of Lemma~\ref{lem:Vevals} hold.
  Then $\|V\|_2 = \alpha$ and $\|V^{-1}\|_2 = \sqrt{n}$.
\end{corollary}
Now, a set of interpolation points are generated using the regular basis. Fix $x$, $h$ and $\eta$. Take $V$ as in \eqref{eq:V}, and define $\{x_1,\dots,x_n,x_1',\dots,x_n'\}$ via
\begin{eqnarray}\label{eq:simplexvspb}
  x_j = x + h v_j, \qquad x_j' = x + \eta h v_j, \qquad  j = 1,\dots, n.
\end{eqnarray}
Note that, because $\|v_j\| = 1$ for all $j$ by \eqref{eq:normvj}, the points $\{x_1,\dots,x_n\}$ all lie on a circle of radius $|h|$, and the points $\{x_1',\dots,x_n'\}$ all lie on a circle of radius $|\eta h|$. Figure~\ref{fig:InterpolationBases} provides an illustration showing the coordinate basis and regular basis in $\R^2$, and how they can be used to generate the interpolation points when $\eta = -1$ (i.e., $h' = -h$).
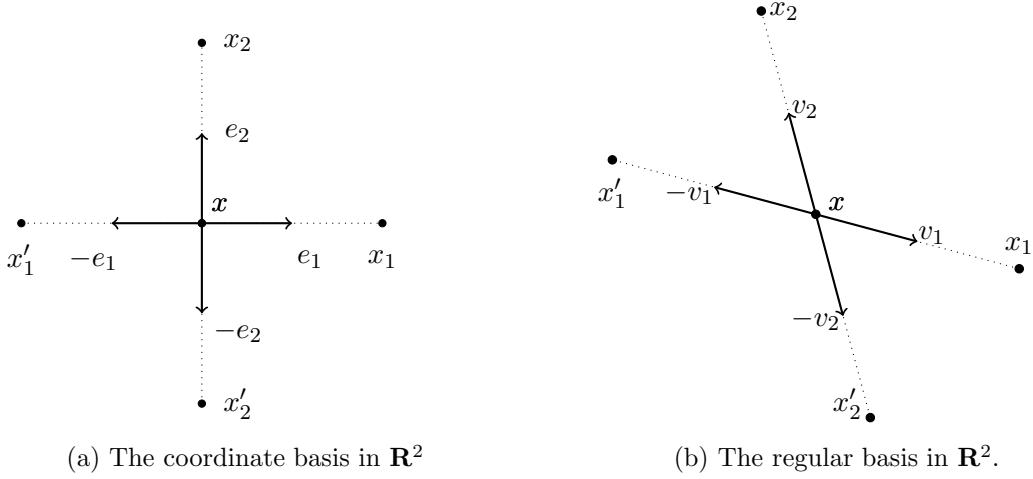
\begin{figure}[H]\centering
\begin{subfigure}[b]{0.4\textwidth}
 \begin{tikzpicture} [scale=2.4]
  \draw[dotted] (-1,0) --(1,0);
  \draw[dotted] (0,-1) --(0,1);
  \draw[thick,<->] (-0.5,0) --(0.5,0);
  \draw[thick,<->] (0,-0.5) --(0,0.5);
  \draw [fill] (0,0)  circle [radius=0.02];
  \node at ( 0.1, 0.1)   {$x$};
  \node at ( 0.6,-0.2)    {$e_1$};
  \node at ( 0.2, 0.5)   {$e_2$};
  \node at (-0.6,-0.2)    {$-e_{1}$};
  \node at ( 0.2,-0.6)   {$-e_{2}$};
  \draw [fill] (0,0)  circle [radius=0.02];
  \draw [fill] (1,0)  circle [radius=0.02];
  \draw [fill] (0,1)  circle [radius=0.02];
  \draw [fill] (-1,0) circle [radius=0.02];
  \draw [fill] (0,-1) circle [radius=0.02];
  \node at ( 0.1, 0.1)   {$x$};
  \node at ( 1.0,-0.2)    {$x_1$};
  \node at ( 0.2, 1.0)   {$x_2$};
  \node at (-1.0,-0.2)    {$x_{1}'$};
  \node at ( 0.2,-1.0)   {$x_{2}'$};
\end{tikzpicture}
\caption{The coordinate basis in $\R^2$}
\label{fig:coordbasis}
\end{subfigure}
\hspace{1cm}
\begin{subfigure}[b]{0.4\textwidth}
  \begin{tikzpicture} [scale=1.4]
  \draw[dotted] (1.9319,-0.5176) --(-1.9319,0.5176);
  \draw[dotted] (-0.5176,1.9319) --(0.5176,-1.9319);
  \draw[thick,<->] (-0.9659,0.2588) --(0.9659,-0.2588);
  \draw[thick,<->] (0.2588,-0.9659) --(-0.2588,0.9659);
  \draw [fill] (0,0) circle [radius=0.04];
  \node at (0.2,0.1) {$x$};
  \node at (1.1,-0.2) {$v_1$};
  \node at (-0.1,1.0) {$v_2$};
  \node at (-1.2,0.2) {$-v_1$};
  \node at (0,-1.0) {$-v_2$};
    \draw [fill] (1.9319,-0.5176) circle [radius=0.04];
  \draw [fill] (-1.9319,0.5176) circle [radius=0.04];
  \draw [fill] (-0.5176,1.9319) circle [radius=0.04];
  \draw [fill] (0.5176,-1.9319) circle [radius=0.04];
  \node at (0.2,0.1) {$x$};
  \node at (1.9319,-0.3) {$x_1$};
  \node at (-1.9319,0.2) {$x_1'$};
  \node at (-0.3,1.9319) {$x_2$};
  \node at (0.3,-1.8) {$x_2'$};
\end{tikzpicture}
\caption{The regular basis in $\R^2$.}
\label{fig:regularbasis}
\end{subfigure}
\caption{Left: The standard coordinate basis $\{e_1,e_2\} \subset\R^2$ and the directions $-e_1,-e_2$. Also the interpolation points $x_j = x + h e_j$, $j\in\{1,2\}$
and $x_j' = x + h' e_j$, $j\in\{1,2\}$ for $h' = -h$. Right: The regular basis $\{v_1,v_2\} \subset\R^2$ and the directions $-v_1,-v_2$. Also the interpolation points $x_j = x + h v_j$, $j\in\{1,2\}$
and $x_j' = x + h' v_j$, $j\in\{1,2\}$ for $h' = -h$.}
\label{fig:InterpolationBases}
\end{figure}

We are now ready to state the main results of this section, which show that estimates of $g(x)$ and $d(x)$ can be obtained in $\oo(n)$ computations when using the model \eqref{eq:modelatxj} and a regular basis \eqref{eq:V}.
\begin{theorem}\label{thm:gRB}
  Let $\alpha$ and $\gamma$ be defined in \eqref{eq:alphagamma}, let $V$ be as defined in \eqref{eq:V}, let $y$ be defined in \eqref{eq:yS}, and let \eqref{ass:sjnotsjp} hold. Then the solution to \eqref{eq:Findgd} is
  \begin{eqnarray}\label{eq:Vg}
 g(x) = \tfrac1{\alpha h}(y + \tfrac1n(\sqrt{n+1} - 1)(y^Te) e),
\end{eqnarray}
which is an $\oo(n)$ computation.
\end{theorem}
\begin{proof}
By Lemma~\ref{lem:Vevals}, $V$ is nonsingular. Because
\begin{eqnarray}\label{eq:gammaineq1}
  1-n\gamma \overset{\eqref{eq:alphagamma}}{=} 1 - n \tfrac1n(1-\tfrac1{\sqrt{n+1}}) = \tfrac1{\sqrt{n+1}},\quad\text{and}\quad \gamma\sqrt{n+1} = \tfrac1n(\sqrt{n+1} - 1),
\end{eqnarray}
applying the Sherman-Morrison-Woodbury formula \eqref{eq:SMW} to $V$ in \eqref{eq:V} gives
\begin{eqnarray}\label{eq:iVtemp}
  V^{-1}= \tfrac1{\alpha}(I + \tfrac{\gamma}{1-n\gamma}ee^T)
  \overset{\eqref{eq:gammaineq1}}{=} \tfrac1{\alpha}(I + \gamma\sqrt{n+1}\; ee^T)
  \overset{\eqref{eq:gammaineq1}}{=} \tfrac1{\alpha}(I + \tfrac1n(\sqrt{n+1} - 1) ee^T).
\end{eqnarray}
Hence, taking $U \equiv V$ in \eqref{eq:Findgd} gives
\begin{eqnarray*}
  g(x) \overset{\eqref{eq:Findgd} }{=} \tfrac1h V^{-1}y
  \overset{\eqref{eq:iVtemp}}{=} \tfrac1{\alpha h}(I + \tfrac1n(\sqrt{n+1} - 1) ee^T)y
  = \tfrac1{\alpha h}(y + \tfrac1n(\sqrt{n+1} - 1)(e^Ty) e ).
\end{eqnarray*}
Lastly, \eqref{eq:Vg} depends upon operations involving vectors only (inner and scalar products and vector additions) so that $g(x)$ is an $\oo(n)$ computation.
\end{proof}
\begin{corollary}\label{cor:gRB}
  Let the assumptions of Theorem~\ref{thm:gRB} hold, and let $\eta = -1$ so that $y$ is as in \eqref{eq:yh}. Setting
    $\tilde f = \frac1n \sum_{i=1}^n (f(x+h v_i) - f(x - hv_i)),$
  gives
  \begin{eqnarray*}
    g(x) = \tfrac1{2 \alpha h} (\ff - \ff' + (\sqrt{n+1}-1) \tilde f e).
  \end{eqnarray*}
  Equivalently, the $i$th component of $g(x)$ is
  \begin{eqnarray}\label{eq:gVbasisCFD}
    [g(x)]_i = \tfrac1{2 \alpha h} (f(x+h v_i) - f(x - hv_i) + (\sqrt{n+1}-1) \tilde f).
  \end{eqnarray}
\end{corollary}
Theorem~\ref{thm:gRB} and Corollary~\ref{cor:gRB} show that using a regular basis in the quadratic interpolation model leads to an estimate for $g(x)$ that is similar to the central finite differences approximation in \eqref{eq:gCFD}. The formula \eqref{eq:gCFD} uses a factor $(2h)^{-1}$, while the formula \eqref{eq:gVbasisCFD} uses a factor $(2\alpha h)^{-1}$. Both \eqref{eq:gCFD} and \eqref{eq:gVbasisCFD} involve the difference between the function values measured along the positive and negative basis directions. However, \eqref{eq:gVbasisCFD} also contains a `correction' term, which involves the function values measured at all the vertices. Lastly, note that Theorem~\ref{thm:gRB} requires knowledge of $f(x)$, but in the special case $\eta = -1$, Corollary~\ref{cor:gRB} shows that $f(x)$ is not used.

Before presenting the next theorem, the following parameters are defined. Let
\begin{equation}\label{eq:muom}
  \mu \eqdef \alpha^2(1-2\gamma), \qquad \omega \eqdef \gamma^2/(1-2\gamma).
\end{equation}
\begin{remark}
  Note that $1-2\gamma$ is monotonically increasing in $n$. Moreover, $n\geq 2$, and when $n=2$, $1 -2\gamma = 1-2\frac12(1-\frac1{\sqrt{3}}) = \frac1{\sqrt{3}}> 0$, so $1-2\gamma \neq 0$ for any $n \geq 2$.
\end{remark}
\begin{theorem}\label{thm:dRB}
  Let $\alpha$ and $\gamma$ be defined in \eqref{eq:alphagamma}, let $\mu$ and $\omega$ be defined in \eqref{eq:muom}, let $V$ be defined in \eqref{eq:V}, let $z$ and $W$ be defined in \eqref{eq:zS} and \eqref{eq:W} respectively, and let \eqref{ass:sjnotsjp} hold. Then the solution to \eqref{eq:Findgd} is
  \begin{equation}\label{eq:dRB}
    d(x) = \tfrac{2}{\mu h^2 }\left(z - \tfrac1n(1-\mu)(e^Tz) e\right),
  \end{equation}
  which is an $\oo(n)$ computation.
\end{theorem}
\begin{proof}
From \eqref{eq:vj} it can be seen that
\begin{equation}\label{vjsquared}
  (v_j)_i^2 = (v_j\hadprod v_j)_i = \begin{cases}
    \alpha^2(1-\gamma)^2 & \text{if } j= i,\\
    \alpha^2\gamma^2 & \text{if } j\neq i.
  \end{cases}
\end{equation}
Thus,
\begin{equation}\label{eq:Wel}
  W \overset{\eqref{vjsquared}}{=} \alpha^2\mat{(1-\gamma)^2&\gamma^2& \dots&\gamma^2\\
            \gamma^2&(1-\gamma)^2& \dots&\gamma^2\\
            \vdots& & \ddots& \vdots \\
            \gamma^2&\gamma^2& \dots&(1-\gamma)^2}  \overset{\eqref{eq:muom}}{=}  \mu(I + \omega ee^T),
\end{equation}
Applying the Sherman-Morrison-Woodbury formula from \eqref{eq:SMW}, to $W$ in \eqref{eq:Wel}, gives
\begin{equation}\label{eq:Winv}
  W^{-1} = \tfrac1{\mu}\left(I - \tfrac{\omega}{1+n\omega} ee^T\right)
  = \tfrac1{\mu}\left(I - \tfrac{\omega n}{1+n\omega} \tfrac1nee^T\right)
  \overset{\eqref{eq:muvsomegan}}{=} \tfrac1{\mu}\left(I - \tfrac{1}n(1-\mu)ee^T\right).
\end{equation}
Thus
\begin{eqnarray*}
  d(x) \overset{\eqref{eq:Findgd}}{=} \tfrac2{h^2} W^{-1}z
  \overset{\eqref{eq:Winv}}{=} \tfrac{2}{\mu h^2 }\left(I - \tfrac{1}n(1-\mu)ee^T\right)z
  =  \tfrac{2}{\mu h^2 }\left(z - \tfrac1n(1-\mu)(e^Tz) e\right).
\end{eqnarray*}
Finally, \eqref{eq:dRB} depends upon operations involving vectors only (inner and scalar products and vector additions) so that determining $d(x)$ is an $\oo(n)$ computation.
\end{proof}

\begin{corollary}\label{cor:dRB}
  Let the assumptions of Theorem~\ref{thm:dRB} hold, and let $\eta = -1$ so that $z$ is as in \eqref{eq:zh}. Setting $\bar f = \tfrac1n\sum_{i=1}^n(f(x+h v_i) + f(x - hv_i) - 2f(x))$
  gives
  \begin{eqnarray*}
    d(x) = \tfrac1{\mu h^2} (\df + \df' + (1-\mu) \bar f(x) e),
  \end{eqnarray*}
  or equivalently, the $i$th component of $g$ is
  \begin{eqnarray}\label{eq:dVbasisCFD}
    [d(x)]_i = \tfrac1{\mu h^2} (f(x+h v_i) + f(x - hv_i) - 2f(x) + (1-\mu) \bar f(x)).
  \end{eqnarray}
\end{corollary}
Theorem~\ref{thm:dRB} and Corollary~\ref{cor:dRB} show that using a regular basis in the quadratic interpolation model leads to an estimate for $d(x)$ that is similar to the central finite differences approximation in \eqref{eq:dCFD}. The formula \eqref{eq:dCFD} uses a factor $1/h^2$, whereas \eqref{eq:dVbasisCFD} uses a factor $1/\mu h^2$. However, \eqref{eq:gVbasisCFD} also contains a `correction' term, which is a kind of weighted average of the function values measured at the interpolation points.

\section{Interpolation using minimal positive bases}
\label{S:interp_minbasis}

Spendley et al. \cite{Spendley+H62} and Nelder and Mead \cite{Nelder1965} carried out some of the pioneering work on derivative-free simplex-based algorithms. This work continues today with studies on properties of simplicies \cite{Audet2017,Nadeau2019} as well as work on simplex algorithm development, \cite{Rios2013}.
These, and many other algorithms in derivative free optimization, employ positive bases and simplices, and involve geometric operations such as expansions, rotations and contractions. In such cases, the original simplex/positive basis involves $n+1$ (or possibly more) points (with their corresponding function values), and if a geometric operation is performed, then this often doubles the number of interpolation points and corresponding function values that are available. As previously mentioned, function values are expensive, so it is prudent to make full use of them.

This section is motivated in part, by the ongoing research on simplex based methods, and the following types of ideas. Consider a simplex based algorithm, where one arrives at an iteration that involves an expansion step. In such a case, one would have access to $2n+2$ function values (measured at the vertices of the 2 resulting original and expanded simplices) as a by-product. If an approximation to the gradient or diagonal of the Hessian was readily available, then this might be useful curvature information that could be built into a simplex based algorithm to help locate the solution more quickly. Moreover, depending on the type of geometric operation involved, the vertices of the 2 simplices are generated from the same set of base directions, one scaled by $h$, and one scaled by some $h' = \eta h$. Therefore, in this work we have not restricted to $\eta = -1$, to ensure the results are applicable to expansions, contractions, and not just rotations.

In this section we investigate the use of minimal positive bases (and their associated simplices) to generate interpolation points that can be used in the model \eqref{eq:modelatxj}. We will show that the results from the previous section carry over, in a natural way, when there are $2n+2$ (rather than $2n$) interpolation points. 

\subsection{Positive Bases and Simplices}
This section begins with several preliminaries, (see also, for example, \cite{Conn09,Coope+P2000,Coope19}).
\label{S:simplex}
\begin{definition}[Positive Basis]
    A set of vectors $\{u_1,\dots,u_N\} \subset \R^n$ is called a positive basis for $\R^n$ if and only if the following two conditions hold.
    \begin{itemize}
      \item[(i)] Every vector in $\R^n$ is a nonnegative linear combination of the members of $\{u_1,\dots,u_N\}$, i.e., $\forall \; u\in \R^n$ it holds that $\{u \in \R^n : u = c_1 u_1 + \dots + c_N u_N, \; 0\leq c_i \in\R, i = 1,\dots, N\}.$
      \item[(ii)] No proper subset of $\{u_1,\dots,u_N\}$ satisfies (i).
    \end{itemize}
  \end{definition}

  \begin{definition}[Minimal and Maximal Positive Bases]
    A positive basis with $n+1$ elements is called a minimal positive basis. A positive basis with $2n$ elements is called a maximal positive basis.
  \end{definition}

\begin{lemma}[Minimal Positive Basis]\label{lem:minposbasisU}
  Let $U=\mat{u_1&\dots&u_n}$ be a nonsingular matrix. Then the set $\{u_1,\dots,u_n,-Ue\}\subset \R^n$ is a minimal positive basis for $\R^n$.
\end{lemma}
Lemma~\ref{lem:minposbasisU} shows that the set $\{e_1,\dots,e_n,-e\} $ is a minimal positive basis for $\R^n$. Lemma~\ref{lem:minposbasisU} also shows that the columns of
\begin{equation}\label{eq:Vp}
  \V \eqdef \mat{V & -Ve}\in \R^{n\times(n+1)},
\end{equation}
where $V$ is defined in \eqref{eq:V}, form a minimal positive basis for $\R^n$ (recall that $V$ is nonsingular by Lemma~\ref{lem:Vevals}). It can be shown that (see also Lemma 3 in \cite{Coope19}),
\begin{equation}\label{eq:vnp1}
  v_{n+1} = -Ve = - \sum_{i=1}^n v_j \overset{\eqref{eq:V}+\eqref{eq:Vp}}{=} - \frac{1}{\sqrt{n}}e.
\end{equation}
The minimal positive basis $\{v_1,\dots,v_{n+1}\}$ has many useful properties and it is shown in \cite{Coope19} that the angle between any pair of elements of the minimal positive basis is the same, while \eqref{eq:vnp1} shows that one of the directions is aligned with the vector $-e$. Henceforth, we refer to the columns of $\V$ in \eqref{eq:Vp} as a \emph{regular minimal positive basis}.

\begin{definition}[Affine independence]
  A set of $n+1$ points $x_1\dots,x_{n+1}\in\R^n$ is called affinely independent if the vectors $x_1-x_{n+1},\dots,x_{n}-x_{n+1}$ are linearly independent.
\end{definition}
\begin{definition}[Simplex]
  The convex hull of a set of affinely independent points $\{x_1\dots,x_{n+1}\}\subset\R^n$ is called a simplex of dimension $n$.
\end{definition}
\begin{definition}[Regular Simplex]
  A regular simplex is a simplex that is also a regular polytope.
\end{definition}
Given a point $x$, and the regular minimal positive basis $\{v_1,\dots,v_{n+1}\}$, the sets of points $\{x_1,\dots,x_{n+1}\}$ and $\{x_1',\dots,x_{n+1}'\}$ formed via
\begin{equation}\label{eq:convhull}
  x_j = x + h v_j, \qquad  x_j' = x + h' v_j, \qquad j = 1,\dots,n+1,
\end{equation}
are affinely independent. Hence, the convex hull of $\{x_1,\dots,x_{n+1}\}$ is a regular simplex aligned in the direction $-e$, and so too is the the convex hull of $\{x_1',\dots,x_{n+1}'\}$. This explains why we have adopted the terminology `regular', and refer to the columns of $V$ as a regular basis, and the columns of $\V$ as a regular minimal positive basis; both are related to a regular simplex (see Section~3.1 in \cite{Coope19}).
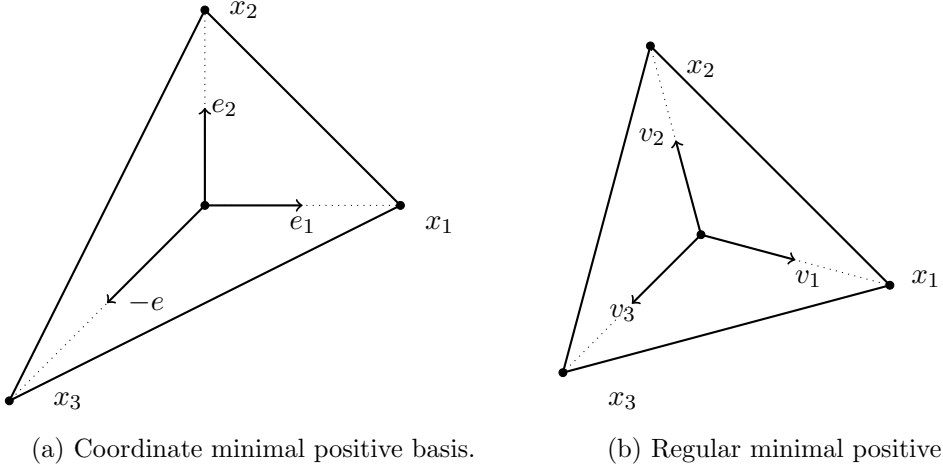
\begin{figure}[H]\centering
\begin{subfigure}[b]{0.4\textwidth}
 \begin{tikzpicture} [scale=1.3]
  \draw[thick,->] (0,0) --(1,0);
  \draw[thick,->] (0,0) --(0,1);
  \draw[thick,->] (0,0) --(-1,-1);
  \draw[dotted] (0,0) --(2,0);
  \draw[dotted] (0,0) --(0,2);
  \draw[dotted] (0,0) --(-2,-2);
  \draw[thick] (2,0) --(0,2);
  \draw[thick] (2,0) --(-2,-2);
  \draw[thick] (-2,-2) --(0,2);
  \draw [fill] (0,0)  circle [radius=0.04];
  \draw [fill] (2,0)  circle [radius=0.04];
  \draw [fill] (0,2)  circle [radius=0.04];
  \draw [fill] (-2,-2)  circle [radius=0.04];
  \node at ( 1.0,-0.2)    {$e_1$};
  \node at ( 0.2, 1.0)    {$e_2$};
  \node at (-0.6,-1)      {$-e$};
  \node at ( 2.4,-0.2)    {$x_1$};
  \node at ( 0.4, 2.0)    {$x_2$};
  \node at (-1.4,-2)      {$x_3$};
\end{tikzpicture}
\caption{Coordinate minimal positive basis.}
\label{fig:minposbasis}
\end{subfigure}
\hspace{5mm}
\begin{subfigure}[b]{0.4\textwidth}
  \begin{tikzpicture} [scale=1.3]
  \draw[thick,->] (0,0) -- (-0.7071,-0.7071);
  \draw[thick,->] (0,0) -- (0.9659,-0.2588);
  \draw[thick,->] (0,0) -- (-0.2588,0.9659);
  \draw[dotted] (0,0) -- (-1.412,-1.412);
  \draw[dotted] (0,0) -- (1.9318,-0.5176);
  \draw[dotted] (0,0) -- (-0.5176,1.9318);
  \draw[thick] (-1.412,-1.412) -- (1.9318,-0.5176);
  \draw[thick] (-0.5176,1.9318) -- (1.9318,-0.5176);
  \draw[thick] (-0.5176,1.9318) -- (-1.412,-1.412);
  \draw [fill] (0,0) circle [radius=0.04];
  \draw [fill] (-1.412,-1.412) circle [radius=0.04];
  \draw [fill] (1.9318,-0.5176) circle [radius=0.04];
  \draw [fill] (-0.5176,1.9318) circle [radius=0.04];
  \node at (1.1,-0.45) {$v_{1}$};
  \node at (-0.5,1)    {$v_{2}$};
  \node at (-0.8,-0.8)   {$v_{3}$};
  \node at (2.3,-0.45) {$x_{1}$};
  \node at (0,1.7)    {$x_{2}$};
  \node at (-0.8,-1.7)   {$x_{3}$};
\end{tikzpicture}
\caption{Regular minimal positive basis}
\label{fig:regularminposbasis}
\end{subfigure}
\caption{An illustration of minimal positive bases in $\R^2$. Left: The coordinate minimal positive basis $\{e_1,e_2,-e\}\subset \R^2$, the points $x_1 = x +h e_1$, $x_2 = x +h e_2$, and $x_3 = x - h e$, as well as the simplex formed by taking the convex hull of $\{x_1,x_2,x_3\}$. Right: The regular minimal positive basis $\{v_1,v_2,v_3\}\subset \R^2$, the points $x_1 = x +h v_1$, $x_2 = x +h v_2$, and $x_3 = x + h v_3$, as well as the simplex formed by taking the convex hull of $\{x_1,x_2,x_3\}$.}
\label{Fig_Simplex_numerics}
\end{figure}

\begin{figure}[h!]\centering
\begin{subfigure}[b]{0.29\textwidth}
  \begin{tikzpicture} [scale=0.5]
  \draw[thick,->] (0,0) -- (-0.7071,-0.7071);
  \draw[thick,->] (0,0) -- (0.9659,-0.2588);
  \draw[thick,->] (0,0) -- (-0.2588,0.9659);
  \draw[dotted] (-3.5355,-3.5355) -- (3,3);
  \draw[dotted] (4.8296,-1.2941) -- (-3.8637 ,1.0353);
  \draw[dotted] (-1.2941,4.8296) -- (1.0353,-3.8637);
  \draw [fill] (0,0) circle [radius=0.02];
  \draw [fill] (-1.4142,-1.4142) circle [radius=0.03];
  \draw [fill] (1.9319,-0.5176) circle [radius=0.03];
  \draw [fill] (-0.5176,1.9319) circle [radius=0.03];
  \draw[thick] (-1.4142,-1.4142) -- (1.9319,-0.5176);
  \draw[thick] (1.9319,-0.5176) -- (-0.5176,1.9319);
  \draw[thick] (-0.5176,1.9319) -- (-1.4142,-1.4142);
  \draw [fill] (-2.8284,-2.8284) circle [radius=0.03];
  \draw [fill] (3.8637,-1.0353) circle [radius=0.03];
  \draw [fill] (-1.0353,3.8637) circle [radius=0.03];
  \draw[thick] (-2.8284,-2.8284) -- (3.8637,-1.0353);
  \draw[thick] (3.8637,-1.0353) -- (-1.0353,3.8637);
  \draw[thick] (-1.0353,3.8637) -- (-2.8284,-2.8284);
  \node at (0.25,0.2) {$x$};
  \node at (4.2,-1.3) {$x_{1}'$};
  \node at (-1.5,4) {$x_{2}'$};
  \node at (-3.3,-2.8) {$x_{3}'$};
  \node at (2.4,-0.7) {$x_{1}$};
  \node at (-0.75,2.2) {$x_{2}$};
  \node at (-1.75,-1.4142) {$x_{3}$};
\end{tikzpicture}
\caption{$0<h<h'$}
\end{subfigure}
\hspace{0.5cm}
\begin{subfigure}[b]{0.29\textwidth}
  \begin{tikzpicture} [scale=0.5]
  \draw[thick,->] (0,0) -- (-0.7071,-0.7071);
  \draw[thick,->] (0,0) -- (0.9659,-0.2588);
  \draw[thick,->] (0,0) -- (-0.2588,0.9659);
  \draw[dotted] (-3.5355,-3.5355) -- (3,3);
  \draw[dotted] (4.8296,-1.2941) -- (-3.8637 ,1.0353);
  \draw[dotted] (-1.2941,4.8296) -- (1.0353,-3.8637);
  \draw [fill] (0,0) circle [radius=0.02];
  \draw [fill] (-1.4142,-1.4142) circle [radius=0.07];
  \draw [fill] (1.9319,-0.5176) circle [radius=0.07];
  \draw [fill] (-0.5176,1.9319) circle [radius=0.07];
  \draw[thick] (-1.4142,-1.4142) -- (1.9319,-0.5176);
  \draw[thick] (1.9319,-0.5176) -- (-0.5176,1.9319);
  \draw[thick] (-0.5176,1.9319) -- (-1.4142,-1.4142);
  \draw [fill] (3.5355,3.5355) circle [radius=0.07];
  \draw [fill] (-4.8296,1.2941) circle [radius=0.07];
  \draw [fill] (1.2941,-4.8296) circle [radius=0.07];
  \draw[thick] (3.5355,3.5355) -- (-4.8296,1.2941);
  \draw[thick] (-4.8296,1.2941) -- (1.2941,-4.8296);
  \draw[thick] (1.2941,-4.8296) -- (3.5355,3.5355);
  \node at (0.25,0.2) {$x$};
  \node at (-5,0.7) {$x_{1}'$};
  \node at (0.5,-4.7) {$x_{2}'$};
  \node at (3.9,2.9) {$x_{3}'$};
  \node at (1.9,-1.1) {$x_{1}$};
  \node at (-0.1,1.9319) {$x_{2}$};
  \node at (-1.7,-1) {$x_{3}$};
\end{tikzpicture}
\caption{$h'<0<h$}
\end{subfigure}
\hspace{0.5cm}
\begin{subfigure}[b]{0.29\textwidth}
  \begin{tikzpicture} [scale=0.5]
  \draw[thick,->] (0,0) -- (-0.7071,-0.7071);
  \draw[thick,->] (0,0) -- (0.9659,-0.2588);
  \draw[thick,->] (0,0) -- (-0.2588,0.9659);
  \draw[dotted] (-3.5355,-3.5355) -- (3,3);
  \draw[dotted] (4.8296,-1.2941) -- (-3.8637 ,1.0353);
  \draw[dotted] (-1.2941,4.8296) -- (1.0353,-3.8637);
  \draw [fill] (0,0) circle [radius=0.02];
  \node at (0.20,0.15) {$x$};
  \draw [fill] (1.4142,1.4142) circle [radius=0.07];
  \draw [fill] (-1.9319,0.5176) circle [radius=0.07];
  \draw [fill] (0.5176,-1.9319) circle [radius=0.07];
  \draw[thick] (1.4142,1.4142) -- (-1.9319,0.5176);
  \draw[thick] (-1.9319,0.5176) -- (0.5176,-1.9319);
  \draw[thick] (0.5176,-1.9319) -- (1.4142,1.4142);
  \draw [fill] (3.5355,3.5355) circle [radius=0.07];
  \draw [fill] (-4.8296,1.2941) circle [radius=0.07];
  \draw [fill] (1.2941,-4.8296) circle [radius=0.07];
  \draw[thick] (3.5355,3.5355) -- (-4.8296,1.2941);
  \draw[thick] (-4.8296,1.2941) -- (1.2941,-4.8296);
  \draw[thick] (1.2941,-4.8296) -- (3.5355,3.5355);
  \node at (-5,0.7) {$x_{1}'$};
  \node at (0.4,-4.7) {$x_{2}'$};
  \node at (4,2.9) {$x_{3}'$};
  \node at (-2,0.9) {$x_{1}$};
  \node at (1.1,-1.9319) {$x_{2}$};
  \node at (1.9,1.2) {$x_{3}$};
\end{tikzpicture}
\caption{$h <h'<0$}
\end{subfigure}
\caption{A schematic of regular simplices and the results of geometric operations (expansion/contraction/rotation). }
\label{Fig_Simplex_geometry}
\end{figure}
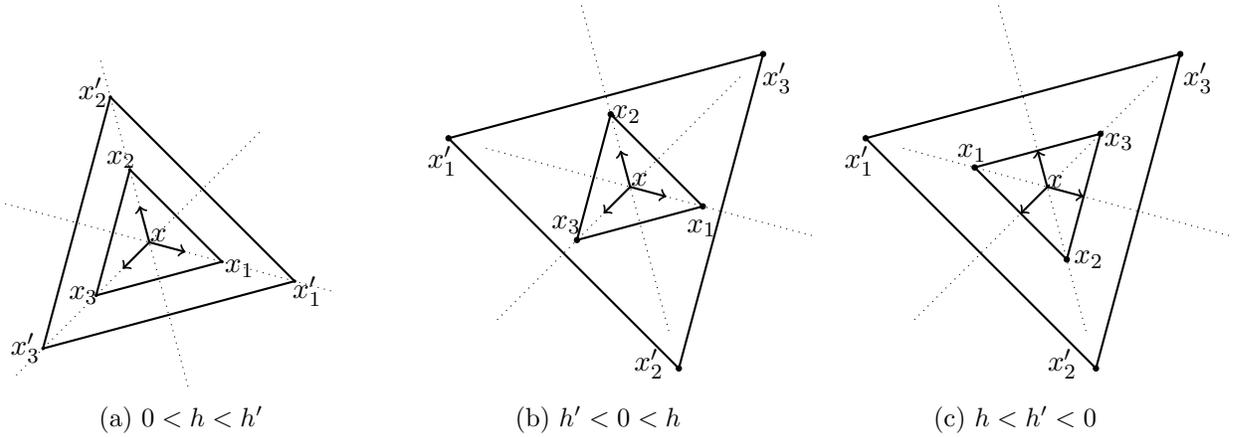

\subsection{Interpolation using a minimal positive basis}
\label{S:interp_regminposbasis}

In this section we show that estimates of the gradient and diagonal Hessian can be obtained in $\oo(n)$ computations when a regular minimal positive basis is used to generate the interpolation points. Hence, let $\{u_1,\dots,u_n\}$ be a basis for $\R^n$, so that $U = \mat{u_1,\dots,u_n}$ is nonsingular. Moreover, let $U_+ = \mat{U & -Ue}$, so that, by Lemma~\ref{lem:minposbasisU} the columns of $\U$ are a minimal positive basis for $\R^n$. Let $W$ be as in \eqref{eq:W} and define
\begin{equation}\label{eq:Wp}
  \W \eqdef \mat{W & u_{n+1} \hadprod u_{n+1}}\in \R^{n \times (n+1)}.
\end{equation}
Following the same arguments as those presented in Section~\ref{S:finitedifferences}, estimates to the gradient and diagonal of the Hessian can be found by solving
\begin{eqnarray}\label{eq:FDgminpb}
  y_+ = h \U^T g(x), \qquad \text{and}\qquad  z_+ = \tfrac{1}2 h^2 \W^T d(x),
\end{eqnarray}
where
\begin{eqnarray}\label{eq:yplus}
   y_+ = \mat{y \\y_{n+1}} \eqdef -\frac1{\eta(1-\eta)}\left(\eta^2\dfp - \dfpp \right)\in \R^{n+1}
\end{eqnarray}
and
\begin{eqnarray}\label{eq:zplus}
  z_+ = \mat{z \\z_{n+1}} \eqdef \frac1{\eta(1-\eta)}\left(\eta\dfp - \dfpp \right)\in \R^{n+1}.
\end{eqnarray}
Clearly, the matrices in \eqref{eq:FDgminpb} are not square, so that $g(x)$ and $d(x)$ will be computed as least squares solutions to \eqref{eq:FDgminpb} in this section.

\subsection{Interpolation using the coordinate minimal positive basis}
Here, take $U = I$, so that $\U = I_+ = \mat{I & -e}$, which corresponds to the coordinate minimal positive basis  $\{e_1,\dots,e_n,-e\}$. The following results hold.

\begin{theorem}\label{thm:gCMPB}
  Let $y_+$ be defined in \eqref{eq:yplus}. Then, using the coordinate minimal basis $\{e_1,\dots,e_n,-e\}$, the least squares solution to \eqref{eq:FDgminpb} is
  \begin{eqnarray*}
    g(x) = \tfrac1h(y - \tfrac1{n+1}(e^T y_+) e).
  \end{eqnarray*}
\end{theorem}
\begin{proof}
Note that
\begin{eqnarray}\label{eq:Iplusinv}
  I_+I_+^T = \mat{I & -e}\mat{I\\-e^T} = I + ee^T, \qquad
  (I_+I_+^T)^{-1} \overset{\eqref{eq:SMW}}{=} I - \tfrac1{n+1}ee^T,
\end{eqnarray}
and
\begin{eqnarray}\label{eq:Iyp}
  I_+ y_+ = \mat{I & -e}\mat{y\\y_{n+1}}  =  y - y_{n+1}e.
\end{eqnarray}
Combining the previous two gives shows that the least squares solution to \eqref{eq:FDgminpb} is
\begin{eqnarray*}
  g(x) &=& \tfrac1h(I_+I_+^T)^{-1} I_+ y_+\\
  &\overset{\eqref{eq:Iplusinv}+\eqref{eq:Iyp}}{=}& \tfrac1h(I - \tfrac1{n+1}ee^T)(y - y_{n+1}e)\\
  &=& \tfrac1h(y - y_{n+1}e  - \tfrac1{n+1}(e^Ty - y_{n+1} n)e)\\
  &=& \tfrac1h(y - \tfrac1{n+1}y_{n+1}e  - \tfrac1{n+1}(e^Ty)e)\\
  &=& \tfrac1h(y  - \tfrac1{n+1}(e^Ty_+)e).
\end{eqnarray*}
\end{proof}
\begin{corollary}\label{cor:gCMPB}
  Let the assumptions of Theorem~\ref{thm:gCMPB} hold. Letting $\eta = -1$ and letting\\ $\tilde{f}_+ = \frac1{n+1}(f(x+he) - f(x-he) + \sum_{i=1}^{n} f(x+he_i) - f(x-he_i))$ gives
\begin{eqnarray*}
    g(x) = \tfrac1{2h}(\ff - \ff' - \tilde{f}_+ e),
  \end{eqnarray*}
  or equivalently
  \begin{eqnarray*}
    [g(x)]_i = \tfrac1{2h}(f(x+he_i) - f(x-he_i) - \tilde{f}_+).
  \end{eqnarray*}
\end{corollary}
Corollary~\ref{cor:gCMPB} shows that when the coordinate minimal positive basis is used in the quadratic interpolation model, the expression for $g(x)$ is similar to the central differences formula, except an additional shift term $\tilde{f}_+$ that involves the function values at all the interpolation points, is also added to each element.
\begin{theorem}\label{thm:dCMPB}
  Let $z_+$ be defined in \eqref{eq:zplus}. Then, using the coordinate minimal positive basis $\{e_1,\dots,e_n,-e\}$, the least squares solution to \eqref{eq:FDgminpb} is
  \begin{eqnarray}\label{eq:dCMPB}
    d(x) = \tfrac2{h^2}(z - \tfrac1{n+1}(e^T z) e + \tfrac1{n+1}z_{n+1}e).
  \end{eqnarray}
\end{theorem}
\begin{proof}
For the coordinate minimal basis  $I_+ = \mat{I&-e}$, so $\W = \mat{I&e}$ because $ -e\hadprod -e = e$. Hence, $\W\W^T =I + ee^T$ and $(\W\W^T)^{-1} =I - \tfrac1{n+1}ee^T$, by \eqref{eq:Iplusinv}. Similarly to \eqref{eq:Iyp}, $\W z_+ = z + z_{n+1}e$.
Combining these shows that the least squares solution to \eqref{eq:FDgminpb} is
\begin{eqnarray*}
   d(x) &=& \tfrac2{h^2}(\W\W^T)^{-1}\W z_+\\
  &\overset{\eqref{eq:Iyp}}{=}& \tfrac2{h^2}(I - \tfrac1{n+1}ee^T)(z + z_{n+1}e)\\
  &=& \tfrac2{h^2}(z + z_{n+1}e  - \tfrac1{n+1}(e^Tz + z_{n+1} n)e)\\
  &=& \tfrac2{h^2}(z - \tfrac1{n+1}(e^Tz)e + \tfrac1{n+1}z_{n+1}e).
\end{eqnarray*}
\end{proof}
\begin{corollary}\label{cor:dCMPB}
  Let the assumptions of Theorem~\ref{thm:gCMPB} hold. Setting $\eta = -1$ and letting\\
  $\bar{f}_+ = \frac1{n+1}\left((f(x+he) + f(x - he)-2f(x)) - \sum_{i=1}^{n} (f(x+he_i) + f(x - he_i)-2f(x))\right)$ gives
  \begin{eqnarray*}
    d(x)  &=& \tfrac1{h^2}(\ff + \ff' - 2f(x)e - \bar{f}_+e),
  \end{eqnarray*}
  or equivalently
  \begin{eqnarray*}
    [d(x)]_i = \tfrac1{h^2}(f(x+he_i) + f(x-he_i) - 2f(x) - \bar{f}_+).
  \end{eqnarray*}
\end{corollary}

\subsection{Interpolation using a regular minimal positive basis}

Here, take $U = V$ with $\U = \V$, where $V$ is defined in \eqref{eq:V}. Furthermore, let $W$ be as in \eqref{eq:W} with
\begin{equation}\label{eq:RMPB_Wp}
  \W \eqdef \mat{W &v_{n+1} \hadprod v_{n+1}} \overset{\eqref{eq:vnp1}}{=} \mat{W &\tfrac1n e}\in \R^{n \times (n+1)}.
\end{equation}
We are now ready to present the main results of this section.
\begin{theorem}\label{thm:gRMPB}
  Let $\alpha$ and $\gamma$ be defined in \eqref{eq:alphagamma}, let $\V$ be the regular minimal positive basis in \eqref{eq:Vp}, let $y_+$ be as in \eqref{eq:yplus}, and let \eqref{ass:sjnotsjp} hold. Then the least squares solution to \eqref{eq:FDgminpb} is
  \begin{eqnarray}\label{eq:gRMPB}
 g(x) = \tfrac{1}{\alpha h}\left(y - (\gamma (e^T y) + \tfrac{1}{\sqrt{n+1}}y_{n+1}) e \right),
\end{eqnarray}
which can be computed in $\oo(n)$.
\end{theorem}
\begin{proof}
Using \eqref{eq:Vp}, it can be shown that $\V\V^T = \alpha^2 I$ (see also Lemma~4 in \cite{Coope19}), so the least squares solution to \eqref{eq:FDgminpb} is $g(x) = \frac1h (\V\V^T)^{-1}\V y_+= \frac1{\alpha^2 h} \V y_+$. Combining with
\begin{eqnarray*}
 \V y_+ &\overset{\eqref{eq:Vp}}{=}& \mat{V & -Ve}\mat{y\\y_{n+1}}\notag\\
  &\overset{\eqref{eq:vnp1}}{=}& Vy - \tfrac{1}{\sqrt{n}}y_{n+1} e\notag\\
  &\overset{\eqref{eq:Vp}}{=}& \alpha (y - \gamma (e^T y) e ) - \tfrac{1}{\sqrt{n}}y_{n+1} e\notag\\
  &=& \alpha (y - (\gamma (e^T y) + \tfrac{1}{\alpha\sqrt{n}}y_{n+1}) e)\notag\\
  &\overset{\eqref{eq:alphagamma}}{=}& \alpha (y - (\gamma (e^T y) + \tfrac{1}{\sqrt{n+1}}y_{n+1}) e),
\end{eqnarray*}
gives \eqref{eq:gRMPB}. Finally, \eqref{eq:gRMPB} depends upon operations involving vectors only; an $\oo(n)$ computation.
\end{proof}
\begin{corollary}\label{cor:gRMPB}
  Let the assumptions of Theorem~\ref{thm:gRMPB} hold. Letting $\eta = -1$ and letting\\ $\tilde{f}_+ = \tfrac{1}{\sqrt{n+1}}(f(x+hv_{n+1}) - f(x-hv_{n+1})) + \gamma \sum_{i=1}^{n} f(x+hv_i) - f(x-hv_i)$ gives
  \begin{eqnarray}\label{eq:FDgminpbLSS}
 g(x) = \tfrac{1}{2\alpha h}\left(\ff - \ff' -  \tilde{f}_+ e \right),
\end{eqnarray}
or equivalently, the $i$th component of $g(x)$ is
\begin{eqnarray}\label{eq:FDgminpbLSS}
 [g(x)]_i = \tfrac{1}{2\alpha h}\left(f(x+hv_i) - f(x-hv_i) -  \tilde{f}_+  \right),
\end{eqnarray}
\end{corollary}

Before presenting the next theorem, define
\begin{equation}\label{eq:sigma}
  \sigma \eqdef 2\omega+\omega^2 n+\tfrac1{\mu^2n^2}.
\end{equation}

\begin{theorem}\label{thm:dRMPB}
  Let $\alpha$ and $\gamma$ be defined in \eqref{eq:alphagamma}, let $\mu$ and $\omega$ be defined in \eqref{eq:muom}, and let $\sigma$ be defined in \eqref{eq:sigma}. Moreover, let $\V$ be the oriented regular simplex defined in \eqref{eq:Vp}, let $z_+$  and $\W$ be defined in \eqref{eq:zplus} and \eqref{eq:RMPB_Wp} respectively, and let \eqref{ass:sjnotsjp} hold. Then least squares solution to \eqref{eq:FDgminpb} is
  \begin{equation}\label{eq:Vd}
    d(x) = \tfrac{2}{\mu h^2 }\left(z  + \frac{1}{1 + \sigma n}\left((\omega-\sigma)(e^T\vz) + \tfrac{1}{\mu n}z_{n+1}\right) e\right),
  \end{equation}
  which can be computed in $\oo(n)$.
\end{theorem}
\begin{proof}
Note that
\begin{eqnarray}\label{eq:WpWpT}
  \W\W^T 
   &\overset{\eqref{eq:Wp}}{=}& WW^T + \tfrac1{n^2}ee^T\notag\\
   &\overset{\eqref{eq:W}}{=}& \mu^2(I+\omega ee^T)^2 + \tfrac1{n^2}ee^T\notag\\
  &=& \mu^2(I+(2\omega+\omega^2 n) ee^T) + \tfrac1{n^2}ee^T\notag\\
  &\overset{\eqref{eq:sigma}}{=}& \mu^2\left(I+\sigma ee^T\right).
\end{eqnarray}
Applying the Sherman-Morrison-Woodbury
formula \eqref{eq:SMW} to \eqref{eq:WpWpT} gives
\begin{equation}\label{eq:WpWpTinv}
  (\W\W^T)^{-1} = \tfrac1{\mu^2}\left(I - \tfrac{\sigma}{1 + \sigma n} ee^T\right).
\end{equation}
Furthermore,
\begin{eqnarray}\label{eq:Wz}
\W \vzp &\overset{\eqref{eq:Wp}+\eqref{eq:zplus}}{=}& W\vz+ \tfrac{1}{n}{z}_{n+1}e\notag\\
  &\overset{\eqref{eq:W}}{=}&   \mu \left(\vz + \omega(e^T\vz)e \right)+ \tfrac{1}{ n} z_{n+1}e \notag\\
  &=& \mu \left(\vz + (\omega(e^T\vz) + \tfrac1{n\mu}z_{n+1})e\right).
\end{eqnarray}
The least squares solution to \eqref{eq:FDgminpb} is
\begin{eqnarray*}
  d &=& \tfrac2{h^2} (\W\W^T)^{-1}\W \vzp\\
  &\overset{\eqref{eq:WpWpTinv}+\eqref{eq:Wz}}{=}& \tfrac{2}{\mu h^2 }\left(I - \tfrac{\sigma}{1 + \sigma n} ee^T\right)\left(\vz + (\omega(e^T\vz) + \tfrac1{n\mu}z_{n+1})e\right)\\
  &=& \tfrac{2}{\mu h^2 }\left(\vz + (\omega(e^T\vz) + \tfrac{z_{n+1}}{\mu n}) e  - \tfrac{\sigma((e^T\vz) + (\omega(e^T\vz) + \tfrac{z_{n+1}}{\mu n}) n) }{1 + \sigma n} e\right)\\
  &=& \tfrac{2}{\mu h^2 }\left(z  + \tfrac{1}{1 + \sigma n}\left((\omega-\sigma)(e^T\vz) + \tfrac{1}{\mu n}z_{n+1}\right) e\right).
\end{eqnarray*}

Finally, \eqref{eq:Vd} depends only upon operations involving vectors only (inner and scalar products and vector additions) so that $d$ is an $\oo(n)$ computation.
\end{proof}
\begin{corollary}\label{cor:dRMPB}
  Let the assumptions of Theorem~\ref{thm:dRMPB} hold. Setting $\eta = -1$ and letting $\bar{f}_+ = \tfrac1{\mu n}(f(x+hv_{n+1}) + f(x - hv_{n+1})-2f(x)) +   (\omega-\sigma)\sum_{i=1}^{n} (f(x+he_i) + f(x - he_i)-2f(x))$ gives
  \begin{eqnarray*}
    d(x)  &=& \tfrac1{h^2}(\ff + \ff' - 2f(x)e - \tfrac{1}{1 + \sigma n}\bar{f}_+e).
  \end{eqnarray*}
\end{corollary}

\subsection{Error bounds for the interpolation model}
\label{S:errorbounds}

The following is a standard error bound; the proof is included for completeness.
\begin{theorem}
  Let $\{u_1,\dots,u_{n+1}\}$ be a minimal positive basis for $\R^n$ and let $h\in\R$. Suppose points $\{x_1,\dots,x_{n+1},x_1',\dots,x_{n+1}'\}$ are constructed via 
\begin{eqnarray}\label{eq:thmpointssetup}
  x_j = x + h u_j, \qquad x_j' = x - h u_j, \qquad j = 1,\dots, n+1.
\end{eqnarray}
  Let $\U = \mat{u_1,\dots,u_{n+1}}$ and suppose that $f$ is continuously differenctiable and that $\nabla^2 f(x)$ is $M$-Lipschitz continuous. Then, the least squares solution $g(x)$ to \eqref{eq:FDgminpb} satisfies
\begin{eqnarray}\label{eq:normbound}
  \|g(x) - \nabla f(x)\|_2 \leq \tfrac16 M h^2 \|\U^\dag\|_2 \sqrt{n+1}.
\end{eqnarray}
\end{theorem}
\begin{proof}
We begin by noticing that
\begin{eqnarray}\label{eq:intermedT}
  (\U\U^T)(g(x) - \nabla f(x) ) &\overset{\eqref{eq:FDgminpb}}{=}& \tfrac1h \U(y_+ - h\U^T\nabla f(x))\notag\\
  &\overset{\eqref{eq:yplus}}{=}& \tfrac1h \U(\tfrac12(\dfp - \dfp') - h\U^T\nabla f(x))
\end{eqnarray}
By \eqref{eq:dfp} and \eqref{eq:thmpointssetup}, the $i$th element of \eqref{eq:intermedT} is
\begin{eqnarray}\label{eq:ithelement}
  [\tfrac12(\dfp - \dfp') - h\U^T\nabla f(x)]_i
  &=& \tfrac12\big((f_j - f(x)) - (f_j' - f(x))\big) - (x_j - x)^T\nabla f(x)\notag\\
  &=& \tfrac12\big(f_j - f_j'\big) - (x_j - x)^T\nabla f(x).
\end{eqnarray}
The integral form of the Mean Value Theorem provides the identities
\begin{eqnarray}
  f_j - f(x) &=& (x_j - x)^T \int_0^1 \nabla f(x + t(x_j - x))dt\label{eq:mvt1}\\
  f_j' - f(x) &=& -(x_j - x)^T \int_0^1 \nabla f(x - t(x_j - x))dt,\label{eq:mvt2}
\end{eqnarray}
because $x_j - x = x - x_j' = -(x_j' - x)$, by \eqref{eq:thmpointssetup}.
By \cite[p.14]{Nocedal2006}, because $f$ is twice continuously differentiable, for $p\in \R^n$,
\begin{eqnarray}\label{eq:MVT2}
  \nabla f(x + p) - \nabla f(x) = \int_0^1 \nabla^2 f(x + \tau p) p \;\; d\tau.
\end{eqnarray}
Combining \eqref{eq:mvt1}, \eqref{eq:mvt2} and \eqref{eq:MVT2}, and subsequently
subtracting $(x_j - x)^T\nabla f(x)$ shows that for each $j=1,\dots,n$:
\begin{eqnarray*}
  &&\tfrac12(f_j - f_j') - (x_j - x)^T\nabla f(x)\\
   &=& \tfrac12(x_j - x)^T \int_0^1 \Big(\nabla f(x + t(x_j - x)) - \nabla f(x) \Big) + \Big(\nabla f(x - t(x_j - x))-\nabla f(x)\Big)dt\\
   &\overset{\eqref{eq:MVT2}}{=}& \tfrac12(x_j - x)^T \int_0^1 \int_0^1 \nabla^2 f(x + \tau t(x_j - x)) t(x_j - x) - \nabla^2 f(x - \tau t(x_j - x))t(x_j - x) dt d\tau\\
   &\leq& \tfrac12\|x_j - x\| \int_0^1 \int_0^1 \|\nabla^2 f(x + \tau t(x_j - x)) t(x_j - x) - \nabla^2 f(x - \tau t(x_j - x))t(x_j - x)\| dt d\tau\\
   &\leq& \tfrac12\|x_j - x\| \int_0^1 \int_0^1 \|\nabla^2 f(x + \tau t(x_j - x)) - \nabla^2 f(x - \tau t(x_j - x))\|\|t(x_j - x)\| dt d\tau\\
   &\leq& \|x_j - x\|^2 \int_0^1 \int_0^1 \|\nabla^2 f(x + \tau t(x_j - x)) - \nabla^2 f(x - \tau t(x_j - x))\|t dt d\tau\\
   &\leq& M \|x_j - x\|^3 \int_0^1 \int_0^1 \tau t^2 dt d\tau\\
   &=& \tfrac{1}{6} M h^3.
\end{eqnarray*}
Hence, taking the 2-norm of \eqref{eq:intermedT}, using the calculation above, and recalling the standard norm inequality $\|\cdot\|_{\infty}\leq \sqrt{n+1}\|\cdot\|_2$ (for vectors in $\R^{n+1}$), gives \eqref{eq:normbound}.
\end{proof}

\subsection{Summary for quadratic interpolation models}
The following tables summarize our results for diagonal quadratic interpolations models. The formulae for $g(x)$ are expressed in terms of $y_+$ defined in \eqref{eq:yplus}. The error bounds give $\kappa$, where $\|g(x) - \nabla f(x)\|\leq \tfrac16Mh^2 \kappa$, and holds when $\eta = -1$.
\begin{table}[!ht]\centering
  \begin{tabular}{|l|c|c|c|c|}
    \hline
     & $g(x) $ & Error $\kappa$ & \# Samples & Theorem\\
    \hline
    & & & & \\[-3mm]
   CB   & $\tfrac1h y$ &$\sqrt{n}  $ & $2n$ & \\[3mm]
   RB   & $\tfrac1{\alpha h}(y + \gamma\sqrt{n+1}(y^Te) e)$&$n$& $2n$   & Thm~\ref{thm:gRB}\\[3mm]
   CMPB & $\tfrac1h(y - \tfrac1{n+1}(e^T y_+) e)$ &$\sqrt{n+1}  $ & $2n+2$ & Thm~\ref{thm:gCMPB}\\[3mm]
   RMPB & $\tfrac{1}{\alpha h}\left(y - (\gamma (e^T y) + \tfrac{1}{\sqrt{n+1}}y_{n+1}) e \right)$  &$\sqrt{n}  $ & $2n+2$  & Thm~\ref{thm:gRMPB}\\[1mm]
   \hline
  \end{tabular}
  \label{table:linearmodelsummary}
  \caption{CB = Coordinate Basis, RB = Regular Basis, CMPB = Coordinate Minimal Positive Basis, RMPB = Regular Minimal Positive Basis. Quadratic model summary for the gradients.}
\end{table}

\begin{table}[!ht]\centering
  \begin{tabular}{|l|c|c|c|}
    \hline
        & $d(x) $  & \# Samples & Theorem\\
    \hline
        &    &  & \\[-3mm]
   CB   &  $\tfrac2{h^2} z$  & $2n$    & \\[3mm]
   RB   &  $\tfrac{2}{\mu h^2 }\left(z - \tfrac1n(1-\mu)(e^Tz) e\right)$   & $2n$    & Thm~\ref{thm:dRB} \\[3mm]
   CMPB &  $\tfrac2{h^2}(z - \tfrac1{n+1}(e^T z) e + \tfrac1{n+1}z_{n+1}e)$  & $2n+2$  & Thm~\ref{thm:dCMPB}\\[3mm]
   RMPB & $\tfrac{2}{\mu h^2 }\left(z  + \frac{1}{1 + \sigma n}\left((\omega-\sigma)(e^T\vz) + \tfrac{1}{\mu n}z_{n+1}\right) e\right)$   & $2n+2$  & Thm~\ref{thm:dRMPB}\\[1mm]
   \hline
  \end{tabular}
  \label{table:linearmodelsummary}
  \caption{CB = Coordinate Basis, RB = Regular Basis, CMPB = Coordinate Minimal Positive Basis, RMPB = Regular Minimal Positive Basis. Quadratic model summary for the diagonal of the Hessian.}
\end{table}

We conclude this section by remarking that, in the special case when $n=2$, if a minimal positive basis is used to generate the points $\{x_1,x_2,x_3,x_1',x_2',x_3'\}$ then there is enough information to uniquely determine all entries of $H(x)$ (recall that $H(x)\approx \nabla^2 f(x))$, because $2n+2 = 6 = (n+1)(n+2)/2$ function values are available, i.e., one can approximate $\nabla^2 f(x)$ rather than only ${\rm diag}(\nabla^2 f(x))$. The algebra is straightforward, so is omitted for brevity.

\section{Linear Interpolation Models}
\label{S:linearmodel}

While quadratic interpolation models lead to accurate approximations to the gradient and diagonal of the Hessian, they require at least $2n$ function evaluations (and more if the full Hessian is desired). For some applications, this cost is too high, and a linear interpolation model is preferred. In this case, one can obtain estimates of the gradient in $\oo(n)$ computations, and this section details how.  As usual, letting $g(x)$ denote an approximation to the gradient $\nabla f(x)$, a linear model is
\begin{equation}\label{eq:linearmodel}
  m(y) = f(x) + (y-x)^T g(x).
\end{equation}
Estimations of the gradient $g(x)$ can be obtained easily from a linear model, although the approximations from such a model (using $n$ or $n+1$ function evaluations depending on the interpolation points) are $\oo(h)$ accurate (in the previous cases with $2n$ or $2n+2$ function evaluations, the approximations are $\oo(h^2)$ accurate).

The interpolation conditions are simply $f(x + hu_j) = m(x+hu_j)$ $\forall j$. If there are $n$ interpolation points (formed from a unit basis for $\R^n)$ then $g(x)$ is the solution to
\begin{equation}\label{eq:linearng}
  \df = h U^T g(x),
\end{equation}
while if there are $n+1$ interpolation points (i.e., a minimal basis for $\R^n)$ then $g(x)$ is the solution to
\begin{equation}\label{eq:linearnpog}
  \dfp = h \U^T g(x).
\end{equation}
An explicit solution to \eqref{eq:linearng} is known for a coordinate basis, and solutions to \eqref{eq:linearnpog} are known for the coordinate minimal positive basis and a regular minimal positive basis.

An explicit solution to \eqref{eq:linearng} for a regular basis (i.e., $V$ in \eqref{eq:V}) is as follows. By \eqref{eq:linearng}, $g(x) = \frac1h V^{-1}\df$. Now, following the same arguments as the proof of Theorem~\ref{thm:gRB} gives
  $g(x) = \frac1{\alpha h} (\df + \gamma \sqrt{n+1}(e^T\df) e)$, or componentwise, for $j = 1,\dots, n$,
\begin{eqnarray*}
  [g(x)]_i = \tfrac1{\alpha h} (f(x+hv_i) - f(x) + \gamma\sqrt{n+1}(e^T\df))
\end{eqnarray*}

The following theorem provides the standard bound for linear interpolation models.
\begin{theorem}\label{thm:linearmodel}
Let $\{u_1,\dots,u_{n+1}\}$ be a minimal positive basis for $\R^n$ and let $h\in\R$. Suppose points $\{x_1,\dots,x_{n+1},x_1',\dots,x_{n+1}'\}$ are constructed via $x_j = x + h u_j$, $x_j' = x + \eta h u_j$ $j = 1,\dots, n+1$. Let $\U = \mat{u_1,\dots,u_{n+1}}$ and suppose that $f$ is continuously differentiable with an $L$-Lipschitz continuous gradient. Then, the least squares solution $g(x)$ to \eqref{eq:linearng} satisfies
\begin{eqnarray}\label{eq:normbound}
  \|g(x) - \nabla f(x)\|_2 \leq \tfrac12 L h \|\U^\dag\|_2 \sqrt{n+1}.
\end{eqnarray}
\end{theorem}
It is clear that if $n$ points are used to generate the interpolation model, then the bound above becomes  $\|g(x) - \nabla f(x)\|_2 \leq \tfrac12 L h \|U^{-1}\|_2 \sqrt{n}$.

The results for linear interpolation models are summarized in Table~\ref{table:linearmodelsummary}. The error bounds give $\kappa_L$, where $\|g(x) - \nabla f(x)\|\leq \tfrac12Lh \kappa_L$, and holds when $\eta = -1$.
\begin{table}[!ht]\centering
  \begin{tabular}{|l|c|c|c|c|}
    \hline
     & g(x) & Error $\kappa_L$ & \# Samples & Ref\\
    \hline
    & & & &\\[-3mm]
   CB   & $\frac1h \df$ & $\sqrt{n}$ &$n$   & \\[3mm]
   RB   & $\frac1{\alpha h} (\df + \gamma \sqrt{n+1}(e^T\df) e)$  & $n$  &$n$   & This work \\[3mm]
   CMPB & $\tfrac1h(\ff - \tfrac1{n+1} (e^T \ffp) e)$ & $\sqrt{n+1}$ &$n+1$ & \cite{Coope+P02} \\[3mm]
   RMPB & $\tfrac{1}{\alpha h}\left(\ff - (\gamma (e^T \ff) - \tfrac{1}{\sqrt{n+1}}f_{n+1}) e \right)$ & $\sqrt{n}$ &$n+1$ & \cite{Coope19} \\[1mm]
   \hline
  \end{tabular}
  \caption{CB = Coordinate Basis, RB = Regular Basis, CMPB = Coordinate Minimal Positive Basis, RMPB = Regular Minimal Positive Basis. Linear model summary.}
  \label{table:linearmodelsummary}
\end{table}

\section{Numerical Experiments}
\label{S:numerics}

Here we present numerical examples to verify the results of this paper, and to provide a concrete application to show that the results have practical use. All experiments are performed using MATLAB (version 2018a).

\subsection{Derivative Computations}\label{S:numres1}

Here we depart from our usual notation and let $y \in \R^2$ with components $y = \mat{y_1&y_2}^T$ so that Rosenbrock's function can be written as
\begin{eqnarray}\label{E_rosenbrock}
  f(y_1,y_2) = (1-y_1)^2+100(y_2-y_1^2)^2.
\end{eqnarray}
 The gradient of \eqref{E_rosenbrock} is
\begin{eqnarray}\label{E_rosenbrock_grad}
  \nabla f(y_1,y_2) = \mat{-2(1-y_1) -400y_1(y_2-y_1^2)\\200(y_2-y_1^2)}
\end{eqnarray}
and the Hessian of \eqref{E_rosenbrock} is
\begin{eqnarray}\label{E_rosenbrock_Hess}
  \nabla^2 f(y_1,y_2) =
  \mat{2 -400y_2 + 1200y_1^2 & -400y_1\\
  -400y_1 & 200}.
\end{eqnarray}
Henceforth, we return to our usual notation.

For the first numerical experiment we simply use the results presented in this work to compute approximations to the gradient and diagonal of the Hessian at two different points, using Rosenbrock's function. The points are chosen to be those where an accurate gradient is needed to make algorithmic progress on Rosenbrock's function.

\paragraph{Point near valley floor}

The first point is
\begin{equation}\label{x0}
  x = \mat{1.1\\1.1^2+10^{-5}},
\end{equation}
which was chosen because it is close to the ``valley floor'' for Rosenbrock's function 
where a good approximation to the gradient is required to make algorithmic progress.

The aligned regular simplex is constructed using the approach previously presented. Here, $n=2$ so that $\alpha = \sqrt{3/2}$ and $\gamma= \frac12(1-\frac1{\sqrt{3}})$ by \eqref{eq:alphagamma}. Recalling that $V = \alpha(I-\gamma ee^T)$ \eqref{eq:V} gives (to 4d.p.)
\begin{eqnarray}\label{NumericalVp}
  \V = \mat{v_1 & v_2 & v_3} = \mat{0.9659&-0.2588&-0.7071\\-0.2588&0.9659&-0.7071}
\end{eqnarray}
For this experiment $h = 10^{-3}$ was used. For the regular basis (RB), the points were generated as $x_j = x + h v_j$, $x_j' = x - h v_j$ for $j = 1,2$, while the regular minimal positive basis (RMPB) used $x_j = x + h v_j$, $x_j' = x - h v_j$ for $j = 1,2,3$. For the coordinate basis (CB), the points were generated as $x_j = x + he_j$, $x_j' = x - h e_j$ for $j = 1,2$, while the coordinate minimal positive basis used the previous points as well as $x_3 = x-he$ and $x_3' = x+he$.

Substituting \eqref{x0} into \eqref{E_rosenbrock_grad} and \eqref{E_rosenbrock_Hess} shows that the analytic gradient and Hessian are
\begin{eqnarray}\label{Grad_x0}
  \nabla f(x) = \mat{0.19559999\\0.00200000},\qquad
  {\rm diag}(\nabla^2 f(x)) =
  10^2 \times
  \mat{ 9.69996000\\2.00000000}
\end{eqnarray}
Approximations to the gradient and Hessian were built using the theorems presented earlier in this work, and are stated in Table~\ref{table:numresults1}. Also stated in the table is
\begin{eqnarray*}
  \epsilon_g = \|g(x) - \nabla f(x) \|_2,\qquad \epsilon_d = \|d(x) - {\rm diag} (\nabla^2 f(x)) \|_2.
\end{eqnarray*}
That is, $\epsilon_g$ is the \emph{computed error} between the true gradient reported in \eqref{Grad_x0} and the computed approximation. Similarly, $\epsilon_d$ is the \emph{computed error} between the true Hessian diagonal reported in \eqref{Grad_x0} and the computed approximation. Note that, for the given point \eqref{x0}, the theoretical error bound for the gradient approximations is $\tfrac16 M h^2 \|\U^{\dag}\|_2\sqrt{n+1}$, where the Lipschitz constant is estimated as $M = 2.3081\times 10^{3}$, so that $\tfrac16 M h^2 \sqrt{n} = 5.4403\times 10^{-4}$.
\begin{table}[h!]\centering
\begin{tabular}{|c|c|c|c|c|}
  \hline
       & $g(x)$ & $\epsilon_g $ & $d(x)$ & $\epsilon_d$ \\
  \hline
       &   &   &   &\\[-2mm]
  CB   & $\mat{0.19603999\\0.00200000}$
       & $4.39\times 10^{-4}$
       & $10^2\times\mat{9.69996199\\1.99999999}$
       & $1.99\times 10^{-4}$\\[9mm]
  RB   & $\mat{0.19608999\\0.00211000}$
       & $5.02\times 10^{-4}$
       & $10^2\times\mat{11.89996197\\4.19999997}$
       & $3.11\times 10^{2}$ \\[9mm]
  CMPB & $\mat{0.19597333\\0.00193333}$
       & $3.79\times 10^{-4}$
       & $10^2\times\mat{6.76662867\\-0.93333333}$
       & $4.15\times 10^{2}$ \\[9mm]
  RMPB & $\mat{0.19592999\\0.00195000}$
       & $3.33\times 10^{-4}$
       & $10^2\times\mat{9.69996175\\1.99999975}$
       & $1.77\times 10^{-4}$ \\[9mm]
  \hline
\end{tabular}
\caption{The computed gradient and diagonal Hessian approximations using the Coordinate Basis, the Regular Basis, the Coordinate Minimal Positive Basis and the Regular Minimal Positive Basis. The computed error in each approximation is also reported.}
\label{table:numresults1}
\end{table}
Table~\ref{table:numresults1} shows the the computed gradient approximations are all close to the true solution \eqref{Grad_x0}. The error in the gradient for each approximation is $\approx 10^{-4}$, which is below the theoretical bound, as expected. For a coordinate basis and a regular minimal positive basis, the approximation to the diagonal of the Hessian is also good, both having error of $\approx 10^{-4}$. However, for a regular basis and a coordinate minimal positive basis, the approximations $d(x)$ are poor, both having large error $\approx 10^{2}$. Upon further inspection, one notices that $d(x)$ for a regular basis is essentially just a shifted version of the approximation for a regular minimal positive basis (subtracting $\approx 2.2\times 10^2$ from each component of the regular basis approximation gives the approximation for a regular minimal positive basis). A similar comment holds for the approximations for a coordinate basis and for a coordinate minimal positive basis. This is consistent with the Theorems presented in this work. For example, comparing Theorems~\ref{thm:dRB} and \ref{thm:dRMPB}, both estimates are the vector $z$, with each component shifted by a fixed amount, with the fixed amount differing for each theorem.

\paragraph{Point near the solution}

Here, the previous experiment is repeated at a point $x$ near the solution $x^* = [1,1]^T$. Here, let
\begin{equation}\label{x0h2}
  x = \mat{0.9\\0.81}.
\end{equation}
When close to the solution, it is appropriate to choose $h$ to be small, so here we set $h = 10^{-6}$. Substituting \eqref{x0h2} into \eqref{E_rosenbrock_grad} and \eqref{E_rosenbrock_Hess} gives
\begin{eqnarray}\label{Grad_x0}
  \nabla f(x) = \mat{-0.20000000\\0},\qquad
  {\rm diag}(\nabla^2 f(x)) =
  10^2 \times
  \mat{ 6.50000000\\2.00000000}.
\end{eqnarray}
For this experiment with the point \eqref{x0h2}, the theoretical error bound for the gradient approximations is again $\tfrac16 M h^2 \|\U^{\dag}\|_2\sqrt{n+1}$, where the Lipschitz constant is estimated as $M = 1.8466\times 10^{3}$, so that $\tfrac16 M h^2 \sqrt{n} = 4.3526\times 10^{-10}$.
\begin{table}[h!]\centering
\begin{tabular}{|c|c|c|c|c|}
  \hline
       & $g(x)$ & $\epsilon_g $ & $d(x)$ & $\epsilon_d$ \\
  \hline
       &   &   &   &\\[-2mm]
  CB   & $\mat{-0.19999999\\0}$
       & $3.54\times 10^{-10}$
       & $10^2\times\mat{6.49999998\\1.99999999}$
       & $1.91\times 10^{-6}$\\[9mm]
  RB   & $\mat{-0.19999999\\0.00000000}$
       & $4.09\times 10^{-10}$
       & $10^2\times\mat{8.30000000\\3.80000003}$
       & $2.55\times 10^{2}$ \\[9mm]
  CMPB & $\mat{-0.19999999\\-0.00000000}$
       & $2.95\times 10^{-10}$
       & $10^2\times\mat{4.09999999\\-0.39999999}$
       & $3.39\times 10^{2}$ \\[9mm]
  RMPB & $\mat{-0.19999999\\-0.00000000}$
       & $2.67\times 10^{-10}$
       & $10^2\times\mat{6.49999999\\2.00000001}$
       & $1.69\times 10^{-6}$ \\[9mm]
  \hline
\end{tabular}
\caption{The computed gradient and diagonal Hessian approximations using the Coordinate Basis, the Regular Basis, the Coordinate Minimal Positive Basis and the Regular Minimal Positive Basis. The computed error in each approximation is also reported.}
\label{table:numresults2}
\end{table}
Table~\ref{table:numresults2} again shows the the computed gradient approximations are all close to the true solution \eqref{Grad_x0}, with the error the gradient approximations begin $\approx 10^{-10}$, all of which are below the theoretical bound. For a coordinate basis and a regular minimal positive basis, the approximation to the diagonal of the Hessian is also good, both having error of $\approx 10^{-6}$. However, for a regular basis and a coordinate minimal positive basis, the approximations $d(x)$ are again poor, both having large error $\approx 10^{2}$. Again, the solution $d(x)$ for a coordinate minimal positive basis is a shifted version of the approximation for a coordinate basis, and similarly for the regular basis and regular minimal positive basis.

\subsection{Derivative estimates within a frame based preconditioned conjugate gradients algorithm}

The purpose of this section is to give an example of how the gradient estimates presented in this work could be used within an existing derivative free optimization algorithm.  The algorithm employed here is the frame based preconditioned conjugate gradients method from \cite{Coope2002} (henceforth referred to as FB-PCG), which can be applied to problems of the form \eqref{eq:prob} where $f$ is a smooth function whose derivatives are unavailable. The algorithm uses a Conjugate Gradients (CG) type inner loop, so approximations to the gradient are needed. It is well known that the practical behaviour of CG often improves when an appropriate preconditioner is used. Thus, the inexpensive estimates $g(x)$ and $D(x)$ are of particular use here; the $D(x)$ developed in this work is an inexpensive approximation to the Hessian, and being diagonal, is a convenient preconditioner. The algorithm in \cite{Coope2002} employs frames as a globalisation strategy to ensure convergence.

Here, the algorithm is implemented \emph{exactly the same way} as in \cite{Coope2002}. The only difference is the way in which the gradient and pure second derivatives are computed. In \cite{Coope2002}, $g(x)$ and $D(x)$ are estimated using finite differences. Here we set $\eta = -1$, and run the algorithm using the four derivative variants presented previously in this work: (i) a Coordinate Basis (CB), which is equivalent to finite difference approximations because $\eta = -1$, (ii) a Regular Basis (RB), (iii) a Coordinate Minimal Positive Basis (CMPB), and (iv) a Regular Minimal Positive Basis (RMPB). It is expected that the practical behaviour of the algorithm will be similar in each instance, because all gradient approximations used are $\oo(h^2)$ accurate. However, we comment that for a coordinate minimal positive basis, the approximation $d(x)$ was computed using finite differences, to ensure the best possible algorithmic performance.

\begin{algorithm}[h!]
	\caption{Frame-Based Preconditioned Conjugate Gradients (FB-PCG)}
	\label{alg:minprp}
	\begin{algorithmic}[1]
        \STATE Initialization: Set $k=1$, $j=n$, $h = \theta_{0} = 1$, $H = I$, $N>0$, $\tau_{\min}>0$, $h_{\min}>0$, $\nu>1$.
		\WHILE {Stopping conditions do not hold}
        \STATE Calculate the function values at the frame points, and form the gradient estimate $g(x_k)$. \IF {$j=1$}
        \STATE Form the pure second derivative estimate $D(x_k)$.
        \ENDIF
		\STATE Check the stopping conditions.
        \STATE Calculate the new search direction $p_{k+1} = - Hg_{k+1} + \beta_k p_k$, where
        \begin{eqnarray}
         \beta_k = \max \{0,\frac{g_k^TH(g_{k+1}-g_k)}{g_k^THg_k}\}, \quad H_{ii} = \frac{1}{\max\{D_i,10^{-4}\}}
        \end{eqnarray}
        \STATE Execute line search: find $\theta_k = \min_{\theta} f(x_k + \theta h_k p_k/\|p_k\|)$ 
        \IF {$j=1$}
        \STATE update $H$, set $x_{k+1}$ to be the point with lowest known function value, set $j = n+3$
        \ELSE
        \STATE decrease $j$ by 1 and set $x_{k+1} = x_k + \theta h_k p_k/\|p_k\|$.
        \ENDIF
		\STATE If frame is quasi-minimal set $h_{k+1} = h_k/\lambda$; else $h_{k+1} = h_k$. Update $k=k+1$.
		\ENDWHILE
	\end{algorithmic}
\end{algorithm}

The algorithm is guaranteed to converge, as confirmed by the following theorem.
\begin{theorem}[Theorem~1 in \cite{Coope2002}]
  If the following conditions are satisfied: (1) the sequence of iterates is bounded; (2) $f$ is continuously differentiable with a locally Lipschitz gradient; and (3) $h_k \to 0$ as $k \to \infty$,
then the subsequence of quasi-minimal frame centers is infinite, and all cluster points of this subsequence are stationary points of $f$.
\end{theorem}

\begin{remark}
  The algorithm of \cite{Liu11} is based upon the FB-PCG algorithm in \cite{Coope2002}, but there are several differences. In \cite{Liu11}, a single minimal positive basis is employed at each iteration (i.e., a linear model is used, and approximations to the second derivatives are unavailable). Also, for the algorithm in \cite{Liu11}, at each iteration, the new point $x_{k+1}$ is always set to be the point with the best (lowest) function value encountered. On the other hand, the Algorithm~\ref{alg:minprp} from \cite{Coope2002} performs $n+3$ iterations of PCG, to maintain the positive properties of the PCG algorithm. After $n+3$ iterations there is a reset, where the PCG process begins again from the point with the lowest function value over the preceding $n+3$ iterations. Because \cite{Liu11} employs only $\oo(h)$ approximations to $g(x)$ and does not use approximations $d(x)$, we do not consider it here.
\end{remark}

The problems tested are taken from the MGH test set in \cite{More1981} and were chosen to allow comparison with the results in \cite{Coope2002} and \cite{Liu11}. The problems are stated in Table~\ref{Table:testproblems}, and the results of applying the FB-PCG algorithms to the test problems are given in Tables~\ref{Table:FBPCGresults1} and \ref{Table:FBPCGresults2}. As suggested in Section~2.2 of \cite{More2009}, each algorithm was terminated if it reached the maximum budget of 1300 function evaluations.
\begin{table}[h!]\centering
\begin{tabular}{|c|l|c|c|l|c|}
\hline
No. & Function Name & $n$ & No. & Function Name & $n$\\
\hline
1  & Rosenbrock                       & 2 & 12 & Box 3-D                & 3\\
2  & Freudenstein and Roth            & 2 & 13 & Powell singular        & 4\\
3  & Powell badly scaled              & 2 & 14 & Woods function         & 4\\
4  & Brown badly scaled               & 2 & 15 & Kowalski and Osborne   & 4\\
5  & Beale                            & 2 & 16 & Brown and Dennis       & 4\\
6  & Jennrich and Sampson             & 2 & 17 & Osborne 1              & 5\\
7  & Helical valley function          & 3 & 18 & Bigg's exponential     & 6\\
8  & Bard                             & 3 & 19 & Osborne 2              & 11\\
9  & Gaussian function                & 3 & 20 & Trigonometric          & 2\\
10 & Meyer function                   & 3 & 21 & Brown almost linear    & 2\\
11 & Gulf research \& development     & 3 & 22 & Broyden tridiagonal    & 20\\
\hline
\end{tabular}
\caption{The test problems used in the numerical experiments.}
\label{Table:testproblems}
\end{table}

\begin{table}[h!]\centering
\begin{tabular}{|c|c|c|c|c|c|c|c|}
  \hline
   Prob. & grad type & nf & fmin & $\|g(x)\|$ & h & itns & qmfs\\
  \hline
  \multirow{4}{*}{1}
   & CB  & 300 & 5.2336e-11  & 7.6674e-06 & 5.6843e-07 &    34 &    15 \\
   & RB  & 1300& 5.3137e+00  & 4.7256e+00 & 1.5625e-02 &   149 &     3 \\
   &CMPB & 345 & 2.8696e-13  & 1.3357e-06 & 3.5527e-06 &    31 &    15 \\
   &RMPB & 381 & 5.5919e-11  & 8.3890e-06 & 2.2737e-06 &    34 &    14 \\
  \hline
  \multirow{4}{*}{2}
   & CB  & 117 & 4.8984e+01  & 8.9198e-05 & 9.5367e-06 &    14 &     9 \\
   & RB  & 102 & 4.8984e+01  & 2.6257e-06 & 3.8147e-06 &    12 &     9 \\
   &CMPB & 245 & 5.9689e-18  & 9.3936e-08 & 5.9605e-07 &    23 &    11 \\
   &RMPB & 165 & 4.8984e+01  & 8.6418e-06 & 5.9605e-06 &    16 &    10 \\
  \hline
  \multirow{4}{*}{3}
   & CB  & 1300 & 1.0178e-08  & 4.8733e+00 & 6.2500e-10 &    86 &    44 \\
   & RB  & 1300 & 1.1101e-08  & 2.2184e+00 & 1.5558e-09 &   102 &    49 \\
   &CMPB & 208  & 1.0101e-08  & 2.0101e-04 & 1.0000e-10 &    18 &    18 \\
   &RMPB & 1300 & 1.5221e-07  & 1.7535e+01 & 1.7937e-08 &    71 &    36\\
  \hline
  \multirow{4}{*}{4}
   & CB  & 188 & 2.8703e-24  & 3.3884e-06 & 9.0949e-08 &    20 &    15 \\
   & RB  & 191 & 4.1550e-26  & 4.0788e-07 & 1.0000e-10 &    21 &    19 \\
   &CMPB & 1300& 1.2157e+09  & 6.9664e+10 & 9.7656e+00 &   119 &     1 \\
   &RMPB & 274 & 7.5858e-24  & 5.5090e-06 & 2.2737e-10 &    21 &    18 \\
  \hline
\multirow{4}{*}{5}
   & CB  & 96  & 1.7741e-12  & 9.5287e-06 & 9.5367e-06 &    12 &     9 \\
   & RB  & 131 & 7.6731e-14  & 7.3698e-07 & 1.4901e-06 &    16 &    11 \\
   &CMPB & 137 & 1.9443e-13  & 7.2973e-07 & 9.5367e-07 &    14 &    10 \\
   &RMPB & 135 & 1.3844e-12  & 1.0291e-06 & 9.5367e-06 &    13 &     9 \\
  \hline
\multirow{4}{*}{6}
   & CB  & 214 & 1.2436e+02  & 1.1450e-06 & 9.3132e-08 &    26 &    13 \\
   & RB  & 179 & 1.2436e+02  & 6.9703e-04 & 2.3842e-06 &    21 &    10 \\
   &CMPB & 242 & 1.2436e+02  & 3.9619e-06 & 1.4901e-07 &    23 &    12 \\
   &RMPB & 198 & 1.2436e+02  & 4.9373e-05 & 5.9605e-08 &    20 &    12 \\
  \hline
\multirow{4}{*}{7}
   & CB  & 277 & 2.4478e-16  & 4.0593e-08 & 2.2737e-08 &    27 &    16 \\
   & RB  & 358 & 1.3523e-12  & 5.2517e-06 & 3.6380e-06 &    33 &    13 \\
   &CMPB & 386 & 1.1919e-17  & 1.2976e-07 & 2.2204e-08 &    31 &    18 \\
   &RMPB & 404 & 1.3124e-13  & 6.2378e-07 & 5.6843e-08 &    32 &    16 \\
  \hline
\multirow{4}{*}{8}
   & CB  & 228 & 8.2149e-03  & 1.5169e-06 & 3.6380e-08 &    21 &    15 \\
   & RB  & 226 & 8.2149e-03  & 1.2996e-06 & 1.4552e-07 &    21 &    14 \\
   &CMPB & 332 & 8.2149e-03  & 3.2610e-06 & 9.3132e-08 &    27 &    13 \\
   &RMPB & 325 & 8.2149e-03  & 5.2576e-06 & 5.6843e-08 &    25 &    16 \\
  \hline
  \multirow{4}{*}{9}
   & CB  & 88  & 1.1279e-08  & 3.0817e-08 & 3.8147e-06 &     9 &     9 \\
   & RB  & 88  & 1.1279e-08  & 2.6819e-08 & 3.8147e-06 &     9 &     9 \\
   &CMPB & 106 & 1.1279e-08  & 2.5277e-08 & 3.8147e-06 &     9 &     9 \\
   &RMPB & 106 & 1.1279e-08  & 2.2890e-07 & 3.8147e-06 &     9 &     9 \\
  \hline
  \multirow{4}{*}{10}
   & CB  & 1300 & 2.6761e+04  & 1.4694e+06 & 1.7937e-03 &    95 &    31 \\
   & RB  & 1300 & 4.7088e+06  & 1.4611e+08 & 8.6736e-05 &   114 &    16 \\
   &CMPB & 1300 & 8.7574e+06  & 3.3756e+08 & 5.5511e-02 &    99 &    12 \\
   &RMPB & 1300 & 6.6179e+06  & 2.0357e+08 & 8.6736e-04 &    95 &    15 \\
  \hline
  \multirow{4}{*}{11}
   & CB  & 611 & 2.1173e-12  & 5.3498e-06 & 1.2622e-08 &    50 &    27 \\
   & RB  & 586 & 8.6109e-10  & 9.3682e-06 & 8.0779e-09 &    48 &    26 \\
   &CMPB & 653 & 5.8997e-10  & 7.3782e-06 & 8.2718e-09 &    46 &    24 \\
   &RMPB & 1300 & 3.7178e-04  & 1.1058e-02 & 7.1746e-05 &    97 &    32\\
  \hline
  \end{tabular}
  \caption{Results of FB-PCG applied to problems 1--11 from Table~\ref{Table:testproblems}.}
\label{Table:FBPCGresults1}
\end{table}
\begin{table}[h!]\centering
\begin{tabular}{|c|c|c|c|c|c|c|c|}
  \hline
  Prob. & op & nf & fmin & ng & h & itn & qmf\\
   \hline
  \multirow{4}{*}{12}
   & CB  & 258 & 1.0126e-06  & 8.9612e-06 & 5.6843e-10 &    22 &    18 \\
   & RB  & 132 & 7.8257e-17  & 1.1635e-08 & 5.9605e-06 &    12 &    10 \\
   &CMPB & 336 & 3.7634e-08  & 5.5972e-06 & 3.5527e-09 &    24 &    18 \\
   &RMPB & 170 & 8.7957e-13  & 9.6558e-08 & 1.4901e-06 &    13 &    11 \\
  \hline
  \multirow{4}{*}{13}
   & CB  & 388 & 9.5087e-09  & 7.8340e-06 & 8.8818e-10 &    29 &    19\\
   & RB  & 680 & 1.2044e-11  & 3.4770e-06 & 8.2718e-10 &    53 &    25 \\
   &CMPB & 539 & 9.7590e-10  & 2.7194e-06 & 1.3878e-09 &    36 &    20 \\
   &RMPB & 732 & 1.2329e-09  & 2.2736e-06 & 5.2940e-09 &    49 &    23 \\
  \hline
  \multirow{4}{*}{14}
   & CB  & 496 & 2.2343e-13  & 3.8564e-06 & 1.3878e-08 &    39 &    19 \\
   & RB  & 664 & 1.0394e-12  & 9.6018e-06 & 1.3553e-08 &    52 &    21 \\
   &CMPB & 1175 & 7.0443e-21  & 3.2592e-09 & 3.1554e-10 &    79 &    29 \\
   &RMPB & 1200 & 2.6560e-16  & 4.1629e-07 & 3.0815e-09 &    80 &    30\\
  \hline
  \multirow{4}{*}{15}
   & CB  & 409 & 3.0751e-04  & 7.0794e-08 & 3.4694e-10 &    29 &    21 \\
   & RB  & 730 & 3.0751e-04  & 8.1424e-07 & 3.0815e-09 &    54 &    30 \\
   &CMPB & 456 & 3.0751e-04  & 2.1624e-07 & 1.0000e-10 &    29 &    22 \\
   &RMPB & 879 & 3.0751e-04  & 2.9915e-06 & 1.9259e-10 &    56 &    32 \\
  \hline
  \multirow{4}{*}{16}
   & CB  & 218 & 8.5822e+04  & 2.0390e-01 & 2.3842e-05 &    18 &     9 \\
   & RB  & 256 & 8.5822e+04  & 5.2632e-02 & 5.9605e-06 &    21 &    10 \\
   &CMPB & 328 & 8.5822e+04  & 8.3039e-03 & 9.5367e-06 &    23 &     9 \\
   &RMPB & 269 & 8.5822e+04  & 1.3964e-02 & 5.9605e-06 &    19 &    10 \\
  \hline
  \multirow{4}{*}{17}
   & CB  & 1300 & 6.1863e-05  & 7.8087e-02 & 7.5232e-06 &    86 &    29 \\
   & RB  & 1300 & 9.2307e-02  & 6.5300e+01 & 1.5259e-03 &    86 &     6 \\
   &CMPB & 1300 & 3.3394e-03  & 6.3918e+00 & 9.5367e-04 &    77 &     7 \\
   &RMPB & 1300 & 5.3120e-02  & 1.9312e+01 & 9.5367e-04 &    78 &     7\\
  \hline
  \multirow{4}{*}{18}
   & CB  & 522  & 5.6556e-03  & 5.2532e-06 & 8.6736e-09 &    30 &    20 \\
   & RB  & 1300 & 5.6557e-03  & 1.8064e-04 & 1.1210e-08 &    73 &    37 \\
   &CMPB & 712  & 5.6557e-03  & 3.1441e-06 & 2.1176e-09 &    36 &    23 \\
   &RMPB & 1300 & 5.6559e-03  & 8.2987e-04 & 7.3468e-08 &    65 &    33 \\
  \hline
  \multirow{4}{*}{19}
   & CB  & 1300 & 4.0437e-02  & 4.3713e-02 & 1.3235e-03 &    48 &    18 \\
   & RB  & 1300 & 2.8701e-01  & 2.5030e+00 & 6.2500e-02 &    49 &     2\\
   &CMPB & 1300 & 4.6094e-02  & 2.3523e-01 & 5.2940e-03 &    45 &    17\\
   &RMPB & 1300 & 2.8656e-01  & 3.4951e+00 & 6.2500e-02 &    46 &     2 \\
  \hline
  \multirow{4}{*}{20}
   & CB  & 80 & 9.2331e-18  & 1.3956e-08 & 3.8147e-06 &    10 &     9 \\
   & RB  & 78 & 1.0496e-17  & 1.4880e-08 & 3.8147e-06 &    10 &     9 \\
   &CMPB & 97 & 1.9978e-17  & 2.0529e-08 & 3.8147e-06 &    10 &     9 \\
   &RMPB & 99 & 8.0511e-18  & 1.3032e-08 & 3.8147e-06 &    10 &     9 \\
  \hline
  \multirow{4}{*}{21}
   & CB  & 77 & 2.3695e-19  & 1.8024e-09 & 3.8147e-06 &    10 &     9 \\
   & RB  & 77 & 1.5376e-16  & 6.2567e-08 & 3.8147e-06 &    10 &     9 \\
   &CMPB & 98 & 6.8671e-14  & 1.3603e-06 & 3.8147e-06 &    10 &     9 \\
   &RMPB & 96 & 7.0015e-15  & 4.3732e-07 & 3.8147e-06 &    10 &     9 \\
  \hline
  \multirow{4}{*}{22}
   & CB  & 1155 & 7.4976e-13  & 7.8573e-06 & 1.0000e-10 &    26 &    19 \\
   & RB  & 1244 & 7.6929e-13  & 8.5455e-06 & 1.0000e-10 &    28 &    18 \\
   &CMPB & 1207 & 2.9752e-13  & 5.3954e-06 & 1.0000e-10 &    26 &    19 \\
   &RMPB & 1113 & 4.3271e-13  & 8.0808e-06 & 1.0000e-10 &    24 &    18 \\
  \hline
\end{tabular}
  \caption{Results of FB-PCG applied to problems 12--22 from Table~\ref{Table:testproblems}.}
\label{Table:FBPCGresults2}
\end{table}

The experiments show that the Frame Based Preconditioned Conjugate Gradients algorithm performs similarly regardless of which gradient and diagonal Hessian approximation is used, as expected.

\section{Conclusion}

This work studied the use of interpolation models to obtain approximations to the gradient and diagonal of the Hessian of a function $f:\R^n\to \R$. We showed that if the coordinate basis is used in the interpolation model, with the parameter $\eta = -1$ used to generate the interpolation points, then one recovers the standard finite difference approximations to the derivative. We also showed that if a regular basis, a coordinate minimal positive basis, or a regular minimal positive basis are used within the interpolation model, then approximations $g(x)$ to the gradient, and $d(x)$ to the pure second derivatives are available in $\oo(n)$ flops. An error bound was presented which shows that for each set of interpolation directions considered in this work, the gradient is $\oo(h^2)$ accurate.
Numerical experiments were presented that show the practical uses of the estimates developed. In particular, we showed how they could be employed within the FB-PCG algorithm to solve \eqref{eq:prob}.

\bibliographystyle{plain}
\bibliography{opt2016}

\begin{thebibliography}{10}

\bibitem{Audet2017}
Charles Audet and Warren Hare.
\newblock {\em Derivative-free and blackbox optimization}.
\newblock Springer Series in Operations Research and Financial Engineering.
  Springer, 2017.

\bibitem{Bandeira2012}
A.~S. Bandeira, K.~Scheinberg, and L.~N. Vicente.
\newblock Computation of sparse low degree interpolating polynomials and their
  application to derivative-free optimization.
\newblock {\em Mathematical Programming Series B}, 134:223--257, 2012.

\bibitem{Baydin2018}
A.~G. Baydin, B.~A. Pearlmutter, A.~A. Radul, and J.~M. Siskind.
\newblock Automatic differentiation in machine learning: a survey.
\newblock {\em Journal of Machine Learning Research}, 18:1--43, 2018.
\newblock Editor: L\'{e}on Bottou.

\bibitem{Berahas2019}
A.~S. Berahas, L.~Cao, K.~Choromanskiy, and K.~Scheinberg.
\newblock A theoretical and empirical comparison of gradient approximations in
  derivative-free optimization.
\newblock Technical report, Department of Industrial and Systems Engineering,
  Lehigh University,, Bethlehem, PA, USA, 2019.
\newblock arXiv:1905.01332v2 [math.OC].

\bibitem{Cocchi2018}
G.~Cocchi, G.~Liuzzi, A.~Papini, and M.~Sciandrone.
\newblock An implicit filtering algorithm for derivative-free multiobjective
  optimization with box constraints.
\newblock {\em Computational Optimization and Applications}, 69(2):267--296,
  2018.

\bibitem{Conn2008}
A.~R. Conn, K.~Scheinberg, and Lu\'{\i}s~N. Vicente.
\newblock Geometry of interpolation sets in derivative free optimization.
\newblock {\em Mathematical Programming Series B}, 111:141--172, 2008.

\bibitem{Conn2008b}
A.~R. Conn, K.~Scheinberg, and Lu\'{\i}s~N. Vicente.
\newblock Geometry of sample sets in derivative-free optimization: polynomial
  regression and underdetermined interpolation.
\newblock {\em IMA Journal of Numerical Analysis}, 28:721--748, 2008.

\bibitem{Conn1996}
Andrew~R. Conn and Philippe~L. Toint.
\newblock An algorithm using quadratic interpolation for unconstrained
  derivative free optimization.
\newblock In G.~Di~Pillo and F.~Giannessi, editors, {\em Nonlinear Optimization
  and Applications}, pages 27--47, Boston, MA, 1996. Springer US.

\bibitem{Conn2009}
A.R. Conn, K.Scheinberg, and L.N. Vicente.
\newblock {\em Introduction to Derivative-Free Optimization}.
\newblock MPS-SIAM Series on Optimization, Philadelphia, 2009.

\bibitem{Conn09}
A.R. Conn, K.Scheinberg, and L.N. Vicente.
\newblock {\em {Introduction to derivative-free optimization}}.
\newblock MPS-SIAM Series on Optimization, Philadelphia, 2009.

\bibitem{Conn1997}
A.R. Conn, K.~Scheinberg, and P.L. Toint.
\newblock Recent progress in unconstrained nonlinear optimization without
  derivatives.
\newblock {\em Mathematical Programming}, 79:397--414, 1997.

\bibitem{Coope2002}
I.~D. Coope and C.~J. Price.
\newblock A derivative-free frame-based conjugate gradients method.
\newblock Ucdms2002-7, University of Canterbury, 2002.
\newblock https://ir.canterbury.ac.nz/handle/10092/11758.

\bibitem{Coope19}
I.~D. Coope and R.~Tappenden.
\newblock Efficient calculation of regular simplex gradients.
\newblock {\em Computational Optimization and Applications}, 72(3):561--588,
  April 2019.

\bibitem{Coope+P2000}
I.D. Coope and C.J. Price.
\newblock {Frame-based methods for unconstrained optimization}.
\newblock {\em Journal of Optimization Theory \& Applications},
  107(2):261--274, 2000.

\bibitem{Coope+P02}
I.D. Coope and C.J. Price.
\newblock {Positive bases in numerical optimization}.
\newblock {\em Computational Optimization \& Applications}, 21:169--175, 2002.

\bibitem{Fasano2009}
Giovanni Fasano, Jos\'{e}~Luis Morales, and Jorge Nocedal.
\newblock On the geometry phase in model-based algorithms for derivative-free
  optimization.
\newblock {\em Optimization Methods and Software}, 24(1):145--154, 2009.

\bibitem{Fazel2018}
Maryam Fazel, Rong Ge, Sham Kakade, and Mehran Mesbahi.
\newblock Global convergence of policy gradient methods for the linear
  quadratic regulator.
\newblock {\em Proceedings of Machine Learning Research (PMLR)}, 80:1467--1476,
  2018.
\newblock International Conference on Machine Learning, 10-15 July 2018,
  Stockholmsm\"{a}ssan, Stockholm, Sweden.

\bibitem{Gilmore1995b}
P.~Gilmore, C.~T. Kelley, C.~T. Miller, and G.~A. Williams.
\newblock Implicit filtering and optimal design problems.
\newblock In Jeffrey Borggaard, John Burkardt, Max Gunzburger, and Janet
  Peterson, editors, {\em Optimal Design and Control}, pages 159--176, Boston,
  MA, 1995. Birkh{\"a}user Boston.

\bibitem{Gilmore1995a}
P.~Gilmore and C.T. Kelley.
\newblock An implicit filtering algorithm for optimization of functions with
  many local minima.
\newblock {\em SIAM J. Optim.}, 5(2):269--285, 1995.

\bibitem{Hare2015}
W.~Hare and M.~Jaberipour.
\newblock Adaptive interpolation strategies in derivative-free optimization: a
  case study.
\newblock Technical report, University of British Colombia, Canada, and
  Amirkabir University of Technology, Iran, November 2015.
\newblock arXiv:1511.02794v1 [math.OC].

\bibitem{Hoffmann2016}
Philipp H.~W. Hoffmann.
\newblock A hitchhiker's guide to automatic differentiation.
\newblock {\em Numerical Algorithms}, 72(3):775--811, 2016.

\bibitem{Nadeau2019}
Gabriel Jarry-Bolduc, Patrick Nadeau, and Shambhavi Singh.
\newblock Uniform simplex of an arbitrary orientation.
\newblock {\em Optimization Letters}, 2019.
\newblock Published online 03 July 2019,
  https://doi.org/10.1007/s11590-019-01448-3.

\bibitem{Liu11}
Q.~Liu.
\newblock {Two minimal positive basis based direct search conjugate gradient
  methods for computationally expensive functions}.
\newblock {\em Numer Algor}, 58:461--474, 2011.

\bibitem{Maggiar2018}
Alvaro Maggiar, Andreas W\"{a}chter, Irina~S Dolinskaya, and Jeremy Staum.
\newblock A derivative-free trust-region algorithm for the optimization of
  functions smoothed via gaussian convolution using adaptive multiple
  importance sampling.
\newblock {\em SIAM Journal on Optimization}, 28(2):1478--1507, 2018.

\bibitem{Margossian2019}
C.~C. Margossian.
\newblock A review of automatic differentiation and its efficient
  implementation.
\newblock Technical report, Department of Statistics, Columbia University,
  January 2019.
\newblock arXiv:1811.05031v2 [cs.MS].

\bibitem{More2009}
J.~J. Mor\'{e} and S.~Wild.
\newblock Benchmarking derivative-free optimization algorithms.
\newblock {\em SIAM Journal on Optimization}, 20(1):172–191, 2009.

\bibitem{More1981}
J.J. More, B.S. Garbow, and K.E. Hillstrom.
\newblock Testing unconstrained optimization software.
\newblock {\em ACM Trans. Math. Softw.}, 7(1):17--41, 1981.

\bibitem{Nelder1965}
J.A. Nelder and R.~Mead.
\newblock A simplex method for function minimization.
\newblock {\em The computer journal}, 7(4):308--313, 1965.

\bibitem{Nesterov2017}
Yurii Nesterov and Vladimir Spokoiny.
\newblock Random gradient-free minimization of convex functions.
\newblock {\em Foundations of Computational Mathematics}, 17(2):527--566, 2017.

\bibitem{Nocedal2006}
Jorge Nocedal and Stephen~J. Wright.
\newblock {\em Numerical Optimization}.
\newblock Springer Series in Operations Research. Springer, 2 edition, 2006.

\bibitem{Rios2013}
Luis~Miguel Rios and Nikolaos~V. Sahinidis.
\newblock Derivative-free optimization: a review of algorithms and comparison
  of software implementations.
\newblock {\em J Glob Optim}, 56:1247--1293, 2013.

\bibitem{Spendley+H62}
W.~Spendley, G.R. Hext, and F.R. Himsworth.
\newblock {Sequential application of simplex designs in optimisation and
  evolutionary operation}.
\newblock {\em Technometrics}, 4:441--461, 1962.

\bibitem{Wild2013}
Stefan~M. Wild and Christine Shoemaker.
\newblock Global convergence of radial basis function trust-region algorithms
  for derivative-free optimization.
\newblock {\em SIAM Review}, 55(2):349--371, 2013.

\end{thebibliography}

\appendix

\section{Constants}

In this section we state several equivalences between constants that are used in this work.
\begin{eqnarray}\label{eq:gammaover1mngamma}
  \tfrac{\gamma}{1 - n\gamma} = \gamma\sqrt{n+1} = \tfrac1n(\sqrt{n+1}-1),
\end{eqnarray}
\begin{eqnarray}\label{eq:gammasquared}
  \gamma^2 \overset{\eqref{eq:alphagamma}}{=}
  \tfrac{1}{n^2}\left(1 - \tfrac1{\sqrt{n+1}}\right)^2= \tfrac{1}{n^2}\left(\tfrac{1}{n+1} + 1 - \tfrac2{\sqrt{n+1}}\right)= \tfrac{1}{n^2}\left(\tfrac{n+2}{n+1} - \tfrac2{\sqrt{n+1}}\right)
\end{eqnarray}
\begin{eqnarray}
  \alpha^2\gamma^2 \overset{\eqref{eq:alphagamma}+\eqref{eq:gammasquared}}{=}
 \tfrac{n+1}{n^3}\left(\tfrac{n+2}{n+1} - \tfrac2{\sqrt{n+1}}\right)
  &=& \tfrac1{n^3}(n+2-2\sqrt{n+1}) \label{eq:alphagammasquared1}\\
  &=& \tfrac1{n^3}(\sqrt{n+1}-1)^2. \label{eq:alphagammasquared2}
\end{eqnarray}
\begin{eqnarray}\label{eq:oneminustwogamma}
  1-2\gamma &=& \left(1 - \tfrac2n + \tfrac2{n\sqrt{n+1}}\right) = \tfrac1n\left(n - 2 + \tfrac2{\sqrt{n+1}}\right) = \tfrac{(n - 2)\sqrt{n+1} + 2}{n\sqrt{n+1}}.
\end{eqnarray}
\begin{eqnarray}\label{eq:muasn}
  \mu = \alpha^2(1-2\gamma) &\overset{\eqref{eq:oneminustwogamma}}{=}& \tfrac{n+1}{n}\left(\tfrac{(n - 2)\sqrt{n+1} + 2}{n\sqrt{n+1}}\right) =  \tfrac1{n^2}\left((n - 2)(n+1) + 2\sqrt{n+1}\right).
\end{eqnarray}
\begin{eqnarray}\label{eq:omegan}
  \omega n  = \tfrac{n\gamma^2}{1-2\gamma}\overset{\eqref{eq:oneminustwogamma}+\eqref{eq:gammasquared}}{=} \left(\tfrac{\sqrt{n+1}}{(n-2)\sqrt{n+1} + 2}\right)\left(\tfrac{n+2 - 2\sqrt{n+1}}{n+1}\right)
  = \left(\tfrac{n+2 - 2\sqrt{n+1}}{(n-2)(n+1) + 2\sqrt{n+1}}\right)
\end{eqnarray}
\begin{eqnarray}\label{eq:oneplusomegan}
  1 + \omega n \overset{\eqref{eq:omegan}}{=} \tfrac{(n-2)(n+1) + 2\sqrt{n+1}}{(n-2)(n+1) + 2\sqrt{n+1}} +\left(\tfrac{n+2 - 2\sqrt{n+1}}{(n-2)(n+1) + 2\sqrt{n+1}}\right) = \tfrac{n^2}{(n-2)(n+1) + 2\sqrt{n+1}}
\end{eqnarray}
\begin{eqnarray}\label{eq:omeganoveroneplusomegan}
  \tfrac{\omega n}{1 + \omega n} \overset{\eqref{eq:omegan}+\eqref{eq:oneplusomegan}}{=} \tfrac1{n^2}\left(n+2 - 2\sqrt{n+1}\right)
\end{eqnarray}
\begin{eqnarray}\label{eq:muvsomegan}
  1 - \mu \overset{\eqref{eq:muasn}}{=}  \tfrac1{n^2}\left(n^2 - \left((n - 2)(n+1) + 2\sqrt{n+1}\right)\right) \overset{\eqref{eq:omeganoveroneplusomegan}}{=} \tfrac{\omega n}{1 + \omega n}
\end{eqnarray}

\end{document}